\documentclass[11pt,a4paper,reqno]{amsart}

\usepackage{amsmath,amsfonts,amssymb}
\usepackage{mathtools}
\usepackage{epsfig}
\usepackage{graphicx}
\usepackage{url}
\usepackage[usenames,dvipsnames]{xcolor}
\usepackage[colorlinks=true,linkcolor=Blue,citecolor=Green,unicode]{hyperref}
\usepackage{nicefrac}
\usepackage{aliascnt}       
\usepackage[alwaysadjust]{paralist}
\usepackage[T1]{fontenc}
\usepackage[utf8]{inputenc}
\usepackage[inline]{enumitem}
\usepackage{framed}
\usepackage{nicefrac}
\usepackage{booktabs}
\setlist[enumerate]{leftmargin=*}
\setlist[itemize]{leftmargin=*}
\usepackage{caption}
\usepackage{textgreek}
\usepackage{blkarray}

\usepackage{algorithm,algorithmic}

\usepackage[textwidth=2.5cm,textsize=small]{todonotes}
\usepackage{cleveref}

\newtheorem{theorem}{Theorem}[section]
\newaliascnt{lemma}{theorem}
\newaliascnt{corollary}{theorem}
\newaliascnt{definition}{theorem}
\newaliascnt{remark}{theorem}
\newaliascnt{proposition}{theorem}
\newaliascnt{conjecture}{theorem}
\newaliascnt{example}{theorem}
\newaliascnt{prob}{theorem}
\newaliascnt{question}{theorem}
\newaliascnt{claim}{theorem}

\newtheorem{lemma}[lemma]{Lemma}
\aliascntresetthe{lemma}

\newtheorem*{lemma*}{Lemma}

\aliascntresetthe{corollary}

\newtheorem*{corollary*}{Corollary}

\newtheorem{definition}[definition]{Definition}
\aliascntresetthe{definition}

\newtheorem*{definition*}{Definition}

\newtheorem{remark}[remark]{Remark}
\aliascntresetthe{remark}

\newtheorem*{remark*}{Remark}

\newtheorem{proposition}[proposition]{Proposition}
\aliascntresetthe{proposition}

\newtheorem*{proposition*}{Proposition}

\newtheorem{conjecture}[conjecture]{Conjecture}
\aliascntresetthe{conjecture}

\newtheorem*{conjecture*}{Conjecture}

\aliascntresetthe{example}

\newtheorem*{example*}{Example}

\newtheorem{prob}[prob]{Problem}
\aliascntresetthe{prob}

\newtheorem*{prob*}{Problem}

\newtheorem{question}[question]{Question}
\aliascntresetthe{question}

\newtheorem*{question*}{Question}

\aliascntresetthe{claim}

\newtheorem*{claim*}{Claim}

\DeclareMathOperator{\supp}{supp}
\DeclareMathOperator{\conv}{conv}
\DeclareMathOperator{\vol}{vol}
\DeclareMathOperator{\pyr}{pyr}
\DeclareMathOperator{\h}{h}
\DeclareMathOperator{\hs}{h_s}
\newcommand\hsd[2]{\operatorname{h_s^{#1}}(#2)}

\def\R{\mathbb{R}}
\def\Z{\mathbb{Z}}

\def\cD{\mathcal{D}}
\def\cE{\mathcal{E}}
\def\cF{\mathcal{F}}
\def\cO{\mathcal{O}}
\def\cP{\mathcal{P}}

\def\fc{\mathfrak{c}}
\DeclarePairedDelimiter{\card}{\lvert}{\rvert}
\newcommand{\one}{\mathbf{1}}
\newcommand{\zero}{\mathbf{0}}
\newcommand{\sfP}{\mathsf{P}}

   {\endMakeFramed}

\begin{document}

\title[On the Maximal Number of Columns of a {\textDelta}-modular Integer Matrix]{On the Maximal Number of Columns of a {\textDelta}-modular Integer Matrix: Bounds and Computations}

\author{Gennadiy Averkov}
\address{BTU Cottbus-Senftenberg\\
  Platz der Deutschen Einheit 1\\
  03046 Cottbus\\
  Germany}
\email{averkov@b-tu.de}

\author{Matthias Schymura}
\address{BTU Cottbus-Senftenberg\\
  Platz der Deutschen Einheit 1\\
  03046 Cottbus\\
  Germany}
\email{schymura@b-tu.de}

\thanks{An extended abstract of this work appeared as~\cite{averkovschymura2022maximal-ipco} at IPCO 2022.}



\date{July 10, 2022}

\begin{abstract}
We study the maximal number of pairwise distinct columns in a $\Delta$-modular integer matrix with~$m$ rows.
Recent results by Lee et al.~provide an asymptotically tight upper bound of $\cO(m^2)$ for fixed~$\Delta$.
We complement this and obtain an upper bound of the form $\cO(\Delta)$ for fixed~$m$, and with the implied constant depending polynomially on~$m$.
\end{abstract}

\maketitle

\section{Introduction}

Full row-rank integer matrices with minors bounded  by a given constant $\Delta$ in the absolute value have been extensively studied in integer linear programming as well as matroid theory:
The interest for optimization was coined by the paper of Artmann, Weismantel \& Zenklusen~\cite{artmannweismantelzenklusen2017astrongly} who showed that integer linear programs with a \emph{bimodular} constraint matrix, meaning that all its maximal size minors are bounded by two in absolute value, can be solved in strongly polynomial time.
With the goal of generalizing the results of Artmann et al.~beyond the bimodular case, N\"agele, Santiago \& Zenklusen~\cite{naegelesantiagozenklusen2021congruency} studied feasibility and proximity questions of a subclass of integer programs with bounded subdeterminants.
Fiorini et al.~\cite{fiorinijoretweltgeyuditsky2021integer} obtained a strongly polynomial-time algorithm for integer linear programs whose defining coefficient matrix has the property that all its subdeterminants are bounded by a constant and all of its rows contain at most two nonzero entries.
For more information on the development regarding this topic, we refer to the three cited contributions above and the references therein. 

For a matrix $A \in \R^{m \times n}$ and for $1 \leq k \leq \min\{m,n\}$, we write
\[
\Delta_k(A) := \max\{ \card{\det(B)} : B \textrm{ is a } k \times k \textrm{ submatrix of } A \}
\]
for the maximal absolute value of a $k \times k$ minor of~$A$.
Given an integer $\Delta \in \Z_{>0}$, a matrix $A \in \R^{m \times n}$ of rank~$m$ is said to be \emph{$\Delta$-modular} and \emph{$\Delta$-submodular}, if $\Delta_m(A) = \Delta$ and~$\Delta_m(A) \leq \Delta$, respectively.\footnote{The authors of~\cite{glanzerweismantelzenklusen2018onthenumber,leepaatstallknechtxu2021polynomial}  use the term $\Delta$-modular for what we call $\Delta$-submodular.}
Moreover, a matrix $A \in \R^{m \times n}$ is said to be \emph{totally $\Delta$-modular} and \emph{totally $\Delta$-submodular}, if $\max_{k \in [m]} \Delta_k(A) = \Delta$ and $\max_{k \in [m]} \Delta_k(A) \leq \Delta$, respectively, where $[m] := \{1,2,\ldots,m\}$.

\noindent Our object of studies is the \emph{generalized Heller constant}, which we define as
\begin{align*}
	\h(\Delta,m) := \max\bigl\{ n \in \Z_{>0} :  A \in \Z^{m \times n} \ & \textrm{has  pairwise distinct columns} \bigr.  \\
	\bigl. & \text{and} \ \Delta_m(A) = \Delta \bigr\}.
\end{align*}
The value $\h(\Delta,m)$ is directly related to the value $\fc(\Delta,m)$ studied in~\cite{glanzerweismantelzenklusen2018onthenumber,leepaatstallknechtxu2021polynomial} and defined as the maximum number~$n$ of columns in a $\Delta$-submodular integer matrix~$A$ with~$m$ rows with the properties that~$A$ has no zero columns and for any two distinct columns $A_i$ and~$A_j$ with $1 \le i <j \le n$ one has $A_i \ne A_j$ and $A_i \ne - A_j$.
It is clear that 
\[
	\fc(\Delta,m) = \frac{1}{2} \bigl( \max \{\h(1,m),\ldots, \h(\Delta,m)\} -1 \bigr),
\]
showing that $\fc(\Delta,m)$ and $\h(\Delta,m)$ are ``equivalent'' in many respects.
However, our proofs are more naturally phrased in terms of $\h(\Delta,m)$ rather than $\fc(\Delta,m)$, as we prefer to prescribe $\Delta_m(A)$ rather than providing an upper bound on $\Delta_m(A)$ and we do not want to eliminate the potential symmetries within $A$ coming from taking columns $A_i$ and $A_j$ that satisfy $A_i  = -A_j$. 

Upper bounds on the number of columns in (totally) $\Delta$-(sub)modular integer matrices with~$m$ rows have been gradually improved over time.
In the case $\Delta = 1$, we are concerned with the notion of (totally) unimodular integer matrices.
The maximal number of pairwise distinct columns in a (totally) unimodular integer matrix  with~$m$ rows has been shown by Heller~\cite{heller1957onlinear} to be equal to $\h(1,m) = m^2 + m + 1$.
Lee~\cite[Sect.~10]{lee1989subspaces} initiated the study of the maximal number of columns beyond unimodular matrices, in 1989, and proved a bound of order $\cO(r^{2\Delta})$, for totally $\Delta$-submodular integer matrices of row-rank~$r$.
Glanzer, Weismantel \& Zenklusen~\cite{glanzerweismantelzenklusen2018onthenumber} revived the story by extending the investigation to $\Delta$-submodular integer matrices and obtaining a polynomial bound in the parameter~$m$.
More precisely, they showed that for each fixed $\Delta \geq 2$, $\h(\Delta,m)$ is of order at most $\cO(\Delta^{2 + \log_2 \log_2 \Delta} \cdot m^2)$.
This result has been recently improved by Lee, Paat, Stallknecht \& Xu~\cite[Thm.~2 \& Prop.~1 \& Prop.~2]{leepaatstallknechtxu2021polynomial} who obtained the exact value
\begin{align}
\h(\Delta,m) = m^2 + m + 1 + 2m(\Delta - 1) \qquad \textrm{if} \qquad \Delta \leq 2 \ \textrm{ or }\ m \leq 2,\label{eqn:lee-et-at-exact}
\end{align}
and, for every $\Delta,m \in \Z_{\geq 3}$, proved the estimates\footnote{Lee et al.~\cite[p.~23]{leepaatstallknechtxu2021polynomial} remark that their techniques provide $\h(\Delta,m) \leq \cO(m^2 \cdot \Delta^{1.95})$.}
\begin{align}
m^2 + m + 1 + 2m(\Delta - 1) \leq \h(\Delta,m) \leq (m^2 + m)\Delta^2 + 1.\label{eqn:lee-et-al-bound}
\end{align}
Bounds on $\h(\Delta,m)$ can also be derived using the machinery of matroid theory.
In their recent work, Geelen, Nelson \& Walsh~\cite[Prop. 8.6.1]{geelennelsonwalsh2021excludingaline} rely on the fact that the class of matroids representable by integer $\Delta$-submodular matrices is minor-closed and that the line on $2 \Delta +2$ points (that is, the uniform matroid of rank two with $2 \Delta +2$ elements) is an excluded minor for $\Delta$-submodular representability.
This shows that $\h(\Delta,m)$ can be bounded by providing a bound, for given positive integers $t$ and $m$, on the size of a simple matroid of rank~$m$ that is representable over real numbers with no $(t+2)$-point line being a minor.
Employing this approach, in~\cite[Thm.~2.2.4]{geelennelsonwalsh2021excludingaline} the bound $\h(\Delta,m) \leq m^2 + f(\Delta) m$ is derived with $f(\Delta)$ being at least double exponential in~$\Delta$ (see the comment in~\cite[p.~3]{leepaatstallknechtxu2021polynomial}). 

The best known upper bounds on $\h(\Delta,m)$ to date have the form of a quadratic polynomial $a(\Delta) m^2 + b(\Delta) m + c$ in $m$, with the coefficients for~$m^2$ and~$m$ possibly depending on~$\Delta$, and the constant term $c \in \R$ being independent of~$\Delta$.
The bounds are incomparable, since for $\Delta \to \infty$, for some results $a(\Delta)$ is large but $b(\Delta)$ is small, while for other bounds it is the other way around.

The lower bound $\h(\Delta,m) \geq m^2 + m + 1 + 2m (\Delta - 1)$ in~\eqref{eqn:lee-et-al-bound} is obtained from the $\Delta$-modular integer matrix with~$m$ rows and whose columns are the elements of the difference set of
\[
\left\{ \zero,e_1,e_2,\ldots,e_m \right\} \cup \left\{ 2e_1,3e_1,\ldots,\Delta e_1\right\},
\]
where~$e_i$ denotes the $i$th coordinate unit vector.
This is a natural generalization of the unimodular matrix ($\Delta = 1$) that attains Heller's result $\h(1,m) = m^2+m+1$.
With this perspective and their precise result~\eqref{eqn:lee-et-at-exact}, for $\Delta \leq 2$ or $m \leq 2$, Lee et al.~\cite{leepaatstallknechtxu2021polynomial} conjecture that the lower bound in~\eqref{eqn:lee-et-al-bound} is actually the correct value of $\h(\Delta,m)$, for any choice of~$\Delta,m \in \Z_{>0}$.

\begin{conjecture}[Lee et al.~\cite{leepaatstallknechtxu2021polynomial}]
\label{conj:lee-exact-value-hDm}
For every $\Delta,m \in \Z_{>0}$, holds
\[
\h(\Delta,m) = m^2 + m + 1 + 2m (\Delta - 1).
\]
\end{conjecture}
\noindent On the qualitative side, Conjecture~\ref{conj:lee-exact-value-hDm} implies that $h(\Delta,m) \le a(\Delta) m^2 + b(\Delta) m+c$ holds with $a(\Delta) \in \cO(1)$ and $b(\Delta) \in \cO(\Delta)$.
If this is true, we would also have $h(\Delta,m) \leq \cO(\Delta) m^2$, but even this estimate has not yet been confirmed since the currently available bounds are asymptotically too large for $\Delta \to \infty$. 
The authors of \cite[p.~24]{leepaatstallknechtxu2021polynomial} ask if there exists a bound of the form $\cO(m^d) \Delta$, for some constant $d \in \Z_{>0}$.
As our main result, we answer this question in the affirmative by showing that a bound of order $\cO(m^4) \Delta$ exists.

\begin{theorem}
\label{thm:polynomial-linear}
Let $\Delta,m \in \Z_{>0}$.
\begin{enumerate}[label=(\roman*)]
 \item If $m \geq 5$, then
 \[
\h(\Delta,m) \leq m^2 + m + 1 + 2\,(\Delta - 1) \cdot \sum_{i=0}^4 \binom{m}{i} \in \cO(m^4) \cdot \Delta.
\]

 \item If $m \geq 4$ and $\Delta$ is odd, then
 \[
\h(\Delta,m) \leq m^2 + m + 1 + 2\,(\Delta - 1) \cdot \sum_{i=0}^3 \binom{m}{i} \in \cO(m^3) \cdot \Delta.
\]
\end{enumerate}
\end{theorem}
\noindent It remains an open question whether our bound can be improved, for all~$\Delta$, to a bound of order $\cO(m^d) \Delta$ for some exponent~$d < 4$. 

Based on computational experiments for small values of~$m$ and~$\Delta$, we found series of counterexamples to \Cref{conj:lee-exact-value-hDm} for~$\Delta \in \{4,8,16\}$.

\begin{theorem}
\label{thm:counterexamples-Delta-4}
We have
\begin{align*}
\h(4,m) &\geq m^2 + 9m - 3    &&\hspace{-1cm}\textrm{ for } \quad m \geq 3, \\
\h(8,m) &\geq m^2 + 19m - 11  &&\hspace{-1cm}\textrm{ for } \quad m \geq 4, \quad \textrm{ and} \\
\h(16,m) &\geq m^2 + 33m - 17 &&\hspace{-1cm}\textrm{ for } \quad m \geq 10.
\end{align*}
\end{theorem}
\noindent These lower bounds exceed the conjecture for~$\h(4,m),\h(8,m)$ and~$\h(16,m)$ by Lee et al.~by the additive terms $2(m-2)$, $4(m-3)$, and $2(m-9)$, respectively.
This means that the qualitative side of \Cref{conj:lee-exact-value-hDm} as described above still stands.

\subsection*{Organisation of the paper}
The paper is organized as follows.
In \Cref{sect:residue-classes}, we describe our geometric idea that explains the linearity in~$\Delta$ of the bounds in \Cref{thm:polynomial-linear}, and we introduce two variants of the Heller constant $\h(1,m)$ which we aim to polynomially bound in \Cref{sec:poly-bound-shifted-heller,sec:sauer-matrix-proof}.
In \Cref{sect:computations}, we describe our approach to compute the generalized Heller constant $\h(\Delta,m)$ for small parameters $\Delta,m$.
We also discuss the results of our computer experiments, and identify counterexamples to \Cref{conj:lee-exact-value-hDm} whose structure lead us to construct the lower bounds in \Cref{thm:counterexamples-Delta-4}. 
Finally, in \Cref{sect:open-problems}, we pose some natural open problems that result from our investigations.

\section{Counting by Residue Classes}
\label{sect:residue-classes}

Our main idea is to count the columns of a $\Delta$-modular integer matrix by residue classes of a certain lattice.
This is the geometric explanation for the linearity in~$\Delta$ of our upper bounds in \Cref{thm:polynomial-linear}.

To be able to count in the non-trivial residue classes, we need to extend the Heller constant $\h(1,m)$ to a \emph{shifted} setting.
Given a translation vector $t \in \R^m$ and a matrix $A \in \R^{m \times n}$, the shifted matrix $t + A := t \one^\intercal + A$ has columns $t + A_i$, where $A_1,\ldots,A_n$ are the columns of~$A$, and $\one$ denotes the all-one vector.

\begin{definition}[Shifted Heller constants]
\label{def:shifted-heller-constant}
Let $m \in \Z_{>0}$ and $\delta \in \Z_{\geq2}$.
\begin{enumerate}[label=(\roman*)]
 \item We define the \emph{shifted Heller constant} $\hs(m)$ as the maximal number~$n$ such that there exists a translation vector $t \in [0,1)^m \setminus \{\zero\}$ and a matrix $A \in \{-1,0,1\}^{m \times n}$ with pairwise distinct columns such that $t+A$ is totally $1$-submodular, that is, $\max_{k \in [m]} \Delta_k(t+A) \leq 1$.

 \item We define the \emph{refined shifted Heller constant} $\hsd{\delta}{m}$ as the maximal number~$n$ such that there exists a vector $t \in [0,1)^m \cap (\tfrac{1}{\delta}\Z)^m \setminus \{\zero\}$ and a matrix $A \in \{-1,0,1\}^{m \times n}$ with pairwise distinct columns such that $t+A$ is totally $1$-submodular, that is, $\max_{k \in [m]} \Delta_k(t+A) \leq 1$.
 
\end{enumerate}
\end{definition}
\noindent Note that, in contrast to the generalized Heller constant $\h(\Delta,m)$, we do not necessarily require $t+A$ to have full rank in the above definition, but we restrict $A$ to have entries in $\{-1,0,1\}$ only.
Compared with $\hs(m)$, in the definition of $\hsd{\delta}{m}$ we only allow the translation vectors~$t$ to have rational coordinates all of whose denominators are divisors of~$\delta$.
Hence, we clearly have $\hsd{\delta}{m} \leq \hs(m)$, for any $\delta \geq 2$.
Moreover, we exclude $t = \zero$ in the definition of both $\hs(m)$ and $\hsd{\delta}{m}$ in order to allow the possibility that $\hsd{\delta}{m} \leq \h(1,m)$, or even $\hs(m) \leq \h(1,m)$.

The reason for restricting the non-zero translation vectors to the half-open unit cube $[0,1)^m$ becomes apparent in the proof of our main lemma.
However, we need to prepare it with an observation on the representation of integer points modulo a sublattice of~$\Z^m$.

\begin{proposition}
\label{prop:coefficients}
Let $\Lambda \subseteq \Z^m$ be a full-dimensional sublattice with basis $b_1,\ldots,b_m \in \Lambda$ and index~$\Delta$.
Then, for every $x \in \Z^m$, the uniquely determined coefficients $\alpha_1,\ldots,\alpha_m$ in the representation
\[
x = \alpha_1 b_1 + \ldots + \alpha_m b_m
\]
satisfy $\alpha_i \in \tfrac{1}{\Delta} \Z$, for every $1 \leq i \leq m$.
\end{proposition}

\begin{proof}
By definition, the finite group $\Z^m  / \Lambda$ has order~$\Delta$.
Hence, the order of any of its elements $x + \Lambda$, where $x \in \Z^m$, divides~$\Delta$.
\end{proof}

\begin{lemma}
\label{lem:residue-class-approach}
For every $\Delta , m \in \Z_{>0}$, we have
\[
\h(\Delta,m) \leq \h(1,m) + (\Delta-1) \cdot \hsd{\Delta}{m}.
\]
In particular, $\h(\Delta,m) \leq \h(1,m) + (\Delta-1) \cdot \hs(m)$.
\end{lemma}

\begin{proof}
Let $A \in \Z^{m \times n}$ be a matrix with $\Delta_m(A) = \Delta$ and  pairwise distinct columns and let $X_A \subseteq \Z^m$ be the set of columns of~$A$.
Further, let $b_1,\ldots,b_m \in X_A$ be such that $\card{\det(b_1,\ldots,b_m)} = \Delta$ and consider the parallelepiped
\[
P_A := [-b_1,b_1] + \ldots + [-b_m,b_m] = \bigg\{ \sum_{i=1}^m \alpha_i b_i : -1 \leq \alpha_i \leq 1, \forall i \in [m]\bigg\} .
\]
Observe that $X_A \subseteq P_A$.
Indeed, assume to the contrary that there is an $x = \sum_{i=1}^m \alpha_i b_i \in X_A$, with, say $|\alpha_j| > 1$.
Then,
\[
\card{\det(b_1,\ldots,b_{j-1},x,b_{j+1},\ldots,b_m)} = \card{\alpha_j} \Delta > \Delta,
\]
which contradicts that $A$ was chosen to be $\Delta$-modular.

Now, consider the sublattice $\Lambda := \Z b_1 + \ldots + \Z b_m$ of $\Z^m$, whose index in~$\Z^m$ equals~$\Delta$.
We seek to bound the number of elements of $X_A$ that fall into a fixed residue class of~$\Z^m$ modulo~$\Lambda$.
To this end, let $x \in \Z^m$ and consider the residue class $x + \Lambda$.
Every element $z \in (x + \Lambda) \cap P_A$ is of the form $z = \sum_{i=1}^m \alpha_i b_i$, for some uniquely determined $\alpha_1,\ldots,\alpha_m \in [-1,1]$ which, in view of \Cref{prop:coefficients}, satisfy $\alpha_i \in \tfrac{1}{\Delta} \Z$, for every $1 \leq i \leq m$.
Moreover, $z$ can be written as
\begin{align}
z &= \sum_{i=1}^m \lfloor \alpha_i \rfloor b_i + \sum_{i=1}^m \{\alpha_i\} b_i,\label{eqn:representation-z-refined}
\end{align}
where $\{\alpha_i\} = \alpha_i - \lfloor \alpha_i \rfloor \in \left\{ 0,\tfrac{1}{\Delta},\ldots,\tfrac{\Delta-1}{\Delta}\right\}$ is the fractional part of $\alpha_i$, and where $\bar x := \sum_{i=1}^m \{\alpha_i\} b_i$ is the unique representative of $x + \Lambda$ in the half-open parallelepiped $[\zero,b_1) + \ldots + [\zero,b_m)$, and in particular, is independent of~$z$.
We use the notation $\lfloor z \rfloor := (\lfloor \alpha_1 \rfloor, \ldots, \lfloor \alpha_m \rfloor) \in \{-1,0,1\}^m$ and $\{ z \} := (\{\alpha_1\},\ldots,\{\alpha_m\}) \in [0,1)^m$ and thus have $z = B (\lfloor z \rfloor + \{ z \})$, where $B = (b_1,\ldots,b_m) \in \Z^{m \times m}$.

Because the vectors $(x + \Lambda) \cap X_A$ constitute a $\Delta$-submodular system and since $\card{\det(b_1,\ldots,b_m)} = \Delta$, the set of vectors $\{ \lfloor z \rfloor + \{ z \} : z \in (x + \Lambda) \cap X_A \}$ are a $1$-submodular system.
For the residue class $\Lambda$, this system is given by $\{ \lfloor z \rfloor : z \in \Lambda \cap X_A \} \subseteq \{-1,0,1\}^m$ and moreover has full rank as it contains $e_1,\ldots,e_m$, and we are thus in the setting of the classical Heller constant $\h(1,m)$.

For the $\Delta - 1$ non-trivial residue classes $x + \Lambda$, $x \notin \Lambda$, we are in the setting of the refined shifted Heller constant $\hsd{\Delta}{m}$.
Indeed, as the matrix with columns $\{b_1,\ldots,b_m\} \cup \left((x + \Lambda)\cap X_A\right) \subseteq X_A$ is $\Delta$-submodular, the matrix with columns
\[
\{e_1,\ldots,e_m\} \cup \left\{ \lfloor z \rfloor + \{ z \} : z \in (x + \Lambda) \cap X_A \right\}
\]
has all its minors, of any size, bounded by~$1$ in absolute value.
By the definition of $\hsd{\Delta}{m}$, the second set in this union has at most $\hsd{\Delta}{m}$ elements.
As a consequence, we get $n = \card{X_A} \leq \h(1,m) + (\Delta-1) \cdot \hsd{\Delta}{m}$, as desired.
\end{proof}

\begin{remark}\ 
\begin{enumerate}[label=(\roman*)]
 \item The proof above shows that we actually want to bound the number of columns~$n$ of a matrix $A \in \{-1,0,1\}^{m \times n}$ such that the system
\[
\{e_1,\ldots,e_m\} \cup \{t+A_1, \ldots, t + A_n\}
\]
is $1$-submodular, for some $t \in [0,1)^m \setminus \{\zero\}$.
However, $t+A$ is totally $1$-submodular if and only if $\{e_1,\ldots,e_m\} \cup (t+A)$ is $1$-submodular.
 
 \item As any matrix $A \in \{-1,0,1\}^{m \times n}$ with pairwise distinct columns can have at most $3^m$ columns, one trivially gets the bound $\hs(m) \leq 3^m$.
Thus, \Cref{lem:residue-class-approach} directly implies the estimate $\h(\Delta,m) \leq 3^m \cdot \Delta$.
\end{enumerate}
\end{remark}

\subsection{Small Dimensions and Lower Bounds in the Shifted Setting}
\label{sec:lower-bounds}

Recall that the original Heller constant is given by $\h(1,m) = m^2 + m + 1$.
The following exact results for dimensions two and three show the difference between this original (unshifted) and the shifted setting grasped by~$\hs(m)$.
Note that as in the shifted setting we require $t \neq \zero$, the Heller constant $\h(1,m)$ is not a lower bound on the shifted Heller constant $\hs(m)$.

\begin{proposition}
\label{prop:shifted-Heller-dim-2-and-3}
We have $\hs(2) = 6$ and $\hs(3) = 12$.
\end{proposition}

\begin{proof}
First, we show that $\hs(2) = 6$.
Let $A \in \{-1,0,1\}^{2 \times n}$ have distinct columns and let $t \in [0,1)^2 \setminus \{\zero\}$ be such that $t+A$ is totally $1$-submodular.
Since $t \neq \zero$, it has a non-zero coordinate, say $t_1 > 0$.
As the $1\times1$ minors of $t+A$, that is, the entries of $t+A$, are bounded in absolute value by~$1$, we get that the first row of $A$ can only have entries in $\{-1,0\}$.
This shows already that $n \leq 6$, as there are simply only $6$ options for the columns of~$A$ respecting this condition.

An example attaining this bound is given by
\[
A =
{\small
\left[\!\begin{array}{rrr|rrr}
-1 & -1 & -1 & 0 & 0 & 0 \\
-1 & 0 & 1 & -1 & \phantom{-}0 & \phantom{-}1
\end{array}\!\right]}
\quad\textrm{ and }\quad
t =
{\small
\begin{bmatrix}
1/2 \\ 0
\end{bmatrix}}.
\]
One can check that (up to permutations of rows and columns) this is actually the unique example $(A,t)$ with $6$ columns in~$A$.

Now, we turn our attention to proving $\hs(3) = 12$.
The lower bound follows by the existence of the following matrix and translation vector
\[
A =
{\small
\left[\!\begin{array}{rrr|rrr|rrr|rrr}
-1 & -1 & -1 & -1 & -1 & -1 &  0 &  0 &  0 &  0 & \phantom{-}0 & \phantom{-}0 \\
-1 & -1 & -1 &  0 &  0 &  0 & -1 & -1 & -1 &  0 & 0 & 0 \\
-1 &  0 &  1 & -1 &  0 &  1 & -1 &  0 &  1 & -1 & 0 & 1
\end{array}\!\right]}
\ ,\ 
t = 
{\small
\begin{bmatrix}
1/2 \\ 1/2 \\ 0
\end{bmatrix}}.
\]
Checking that $t+A$ is indeed totally $1$-submodular is a routine task that we leave to the reader.

For the upper bound, let $A \in \{-1,0,1\}^{3 \times n}$ and $t \in [0,1)^3 \setminus \{\zero\}$ be such that $t+A$ is totally $1$-submodular.
Let~$s$ be the number of non-zero entries of~$t \neq \zero$.
Just as we observed for $\hs(2)$, we get that there are~$s \geq 1$ rows of~$A$ only containing elements from~$\{-1,0\}$.
Thus, if $s=3$ there are only $2^3 = 8$ possible columns and if $s=2$, there are only $2^2 \cdot 3 = 12$ possible columns, showing that $n \leq 12$ in both cases.

We are left with the case that $s=1$, and we may assume that~$A$ has no entry equal to~$1$ in the first row and that $t_1 > 0$.
Assume for contradiction that $n \geq 13$.
There must be $\ell \geq 7$ columns of~$A$ with the same first coordinate, which we subsume into the submatrix~$A'$.
By the identity $\h(1,2) = 7$ applied to the last two rows, and $t_2 = t_3 = 0$, we must have $\ell = 7$ and up to permutations and multiplication of any of the last two rows by $-1$, $A' =
{\tiny \left[\!\begin{array}{*{7}r}
 a & a & a &  a &  a &  a &  a \\
 0 & 1 & 0 & -1 &  0 &  1 & -1 \\
 0 & 0 & 1 &  0 & -1 & -1 &  1
\end{array}\!\right]}$,
for some $a \in \{-1,0\}$.
Since the absolute values of the $2 \times 2$ minors of $t+A$ are bounded by~$1$, the remaining $n-\ell \geq 6$ columns of~$A$ are different from $(b,1,1)^\intercal$ and $(b,-1,-1)^\intercal$, where $b$ is such that $\{a,b\} = \{-1,0\}$.
Under these conditions, we find that~$A$ contains either
$B = {\tiny \left[\!\begin{array}{*{3}r}
-1 & -1 &  0 \\
 1 &  0 &  1 \\
 0 & -1 & -1
\end{array}\!\right]}$,
$B' = {\tiny \left[\!\begin{array}{*{3}r}
-1 & -1 &  0 \\
-1 &  0 & -1 \\
 0 &  1 &  1
\end{array}\!\right]}$,
$C = {\tiny \left[\!\begin{array}{*{3}r}
 0 &  0 & -1 \\
 1 &  0 &  1 \\
 0 & -1 & -1
\end{array}\!\right]}$
or
$C' = {\tiny \left[\!\begin{array}{*{3}r}
 0 &  0 & -1 \\
-1 &  0 & -1 \\
 0 &  1 &  1
\end{array}\!\right]}$
as a submatrix.
However, both the conditions $\card{\det(t+B)} \leq 1$ and $\card{\det(t+B')} \leq 1$ give $t_1 \geq 1$, and both $\card{\det(t+C)} \leq 1$ and $\card{\det(t+C')} \leq 1$ give $t_1 \leq 0$.
Hence, in either case we get a contradiction to the assumption that $0 < t_1 < 1$.
%
\end{proof}

Combining \Cref{lem:residue-class-approach}, the identity $\h(1,m) = m^2 + m + 1$, and \Cref{prop:shifted-Heller-dim-2-and-3} yields the bounds $\h(\Delta,2) \leq 6 \Delta + 1$ and $\h(\Delta,3) \leq 12 \Delta + 1$.
The latter bound improves upon the bound~\eqref{eqn:lee-et-al-bound} of Lee et al.
However, as $\h(\Delta,2) = 4 \Delta + 3$ by~\eqref{eqn:lee-et-at-exact}, we see that the approach via the shifted Heller constant~$\hs(m)$ cannot give optimal results for all~$m$.

A quadratic lower bound on $\hs(m)$ can be obtained as follows:

\begin{proposition}
\label{prop:shifted-heller-lower-bound}
For every $m \in \Z_{>0}$, we have
\[
\hs(m) \geq \h(1,m-1) = m(m-1)+1.
\]
\end{proposition}

\begin{proof}
Let $A' \in \{-1,0,1\}^{(m-1) \times n}$ be a totally unimodular matrix with $n = \h(1,m-1)$ columns, and let $A \in \{-1,0,1\}^{m \times n}$ be obtained from~$A'$ by simply adding a zero-row as the first row.
Then, for the translation vector $t = (\frac{1}{m},0,\ldots,0)^\intercal$ the matrix $t+A$ is totally $1$-submodular.

Indeed, we only need to look at its $k \times k$ minors, for $k \leq m$, that involve the first row, as~$A'$ is totally unimodular by choice.
But then, the triangle inequality combined with developing the given minor by the first row, shows that its absolute value is bounded by~$1$.
\end{proof}

The proof shows that the same lower bound holds for the refined shifted Heller constant $\hsd{\mathit{m}}{m}$.
Note that the denominators of the allowed translation vectors for this constant depend on~$m$ though.

\section{Polynomial Upper Bounds on \texorpdfstring{$\hsd{\Delta}{m}$}{h\_s\textasciicircum\textDelta(m)} and \texorpdfstring{$\hs(m)$}{h\_s(m)}}
\label{sec:poly-bound-shifted-heller}

An elegant and alternative proof for Heller's result that $\h(1,m) = m^2 + m + 1$ has been suggested by Bixby \& Cunningham~\cite{bixbycunningham1987short} and carried out in detail in Schrijver's book~\cite[\S~21.3]{schrijver1986theory}.
They first reduce the problem to consider only the supports of the columns of a given (totally) unimodular matrix and then apply Sauer's Lemma from extremal set theory that guarantees the existence of a large cardinality set that is shattered by a large enough family of subsets of~$[m]$.

We show that this approach can in fact be adapted for the shifted Heller constant $\hs(m)$.
The additional freedom in the problem that is introduced by the translation vectors $t \in [0,1)^m \setminus \{\zero\}$ makes the argument a bit more involved, but still gives a low degree polynomial bound.
To this end, we write $\supp(y) := \left\{ j \in [m] : y_j \neq 0 \right\}$ for the \emph{support} of a vector $y \in \R^m$ and
\[
\cE_A := \left\{ \supp(A_i) : i \in [n] \right\} \subseteq 2^{[m]}
\]
for the family of supports in a matrix~$A \in \R^{m \times n}$ with columns $A_1,\ldots,A_n$.
We use the notation~$2^Y$ for the power set of a finite set~$Y$.

Just as in the unshifted Heller setting, each support can be realized by at most two columns of~$A$, if there exists a translation vector $t \in [0,1)^m$ such that $t+A$ is totally $1$-submodular.

\begin{proposition}
\label{prop:shifted-heller-supports}
Let $A \in \{-1,0,1\}^{m \times n}$ and $t \in [0,1)^m$ be such that $\Delta_k(t + A) \leq 1$, for $k \in \{1,2\}$.
Then, each $E \in \cE_A$ is the support of at most two columns of~$A$.
\end{proposition}

\begin{proof}
Observe that in view of the condition $\Delta_1(t+A) \leq 1$ and the assumption that $t_i \geq 0$, for every $i \in [m]$, we must have $t_r = 0$, as soon as there is an entry equal to~$1$ in the $r$th row of~$A$.

Now, assume to the contrary that there are three columns $A_i,A_j,A_k$ of~$A$ having the same support $E \in \cE_A$.
Then, clearly $\card{E} \geq 2$ and the restriction of the matrix~$(A_i,A_j,A_k) \in \{-1,0,1\}^{m \times 3}$ to the rows indexed by~$E$ is a $\pm 1$-matrix.
Also observe that there must be two rows $r,s \in E$ so that~$(A_i,A_j,A_k)$ contains an entry equal to~$1$ in both of these rows.
Indeed, if there is at most one such row, then the columns~$A_i$, $A_j$, $A_k$ cannot be pairwise distinct.
Therefore, we necessarily have $t_r = t_s = 0$.
Now, there are two options.
Either two of the columns $A_i,A_j,A_k$ are such that their restriction to the rows~$r,s$ give linearly independent $\pm 1$-vectors.
This however would yield a $2 \times 2$ submatrix of $t+A$ with minor $\pm 2$, contradicting that $\Delta_2(t+A) \leq 1$.
In the other case, the restriction of the three columns to the rows~$r,s$ has the form $\pm{\tiny\begin{bmatrix}1 & 1 & 1 \\ 1 & 1 & 1 \end{bmatrix}}$ or $\pm{\tiny\begin{bmatrix}1 & 1 & -1 \\ 1 & 1 & -1\end{bmatrix}}$, up to permutation of the indices~$i,j,k$.
If $\card{E} = 2$, then this cannot happen as~$A$ is assumed to have pairwise distinct columns.
So, $\card{E} \geq 3$, and considering the columns, say $A_i,A_j$, which agree in the rows~$r,s$, there must be another index $\ell \in E \setminus \{r,s\}$ such that $(A_i)_\ell = 1$ and $(A_j)_\ell = -1$, or vice versa.
In any case this means that also $t_\ell = 0$ and that there is a $2 \times 2$ submatrix of~$t+A$ in the rows~$r,\ell$ consisting of linearly independent $\pm 1$-vectors.
Again this contradicts that $\Delta_2(t+A) \leq 1$, and thus proves the claim.
\end{proof}

As mentioned above, this observation on the supports allows to use \emph{Sauer's Lemma} from extremal set theory which we state for the reader's convenience.\footnote{Sauer's Lemma was used already by Glanzer, Weismantel \& Zenklusen~\cite[Lem.~3.4]{glanzerweismantelzenklusen2018onthenumber} for the sake of bounding the number of columns in a $\Delta$-submodular matrix.}
It was independently published by Sauer~\cite{sauer1972onthedensity} and Shelah~\cite{shelah1972acombinatorial} (who also credits M.~Perles) in 1972, and again independently by Vapnik \& Chervonenkis~\cite{vapnikchervonenkis1971onthe} a few years earlier.

\begin{lemma}
\label{lem:sauer-lemma}
Let $m,k \in \Z_{>0}$ be such that $m > k$.
If $\cE \subseteq 2^{[m]}$ is such that
$
\card{\cE} > \binom{m}{0} + \binom{m}{1} + \ldots + \binom{m}{k},
$
then there is a subset $Y \subseteq [m]$ with $k+1$ elements that is \emph{shattered by~$\cE$}, meaning that $\left\{ E \cap Y : E \in \cE \right\} = 2^Y$.
\end{lemma}

Now, the strategy to bounding the number of columns in a matrix $A \in \{-1,0,1\}^{m \times n}$ such that $t+A$ is totally $1$-submodular for some $t \in [0,1)^m$ is to use the inequality $|\cE_A| \geq \frac12 n$, which holds by \Cref{prop:shifted-heller-supports}, and then to argue by contradiction.
Indeed, if $n > 2 \sum_{i=0}^{k-1} \binom{m}{i}$, then by Sauer's Lemma there would be a $k$-element subset $Y \subseteq [m]$ that is shattered by~$\cE_A$.
In terms of the matrix~$A$, this means that (possibly after permuting rows or columns) it contains a submatrix of size $k \times 2^k$ which has exactly one column for each of the~$2^k$ possible supports and where in each column the non-zero entries are chosen arbitrarily from~$\{-1,1\}$.
For convenience we call any such matrix a \emph{Sauer Matrix} of size~$k$.
For concreteness, a Sauer Matrix of size~$3$ is of the form
\[
\left[\!\begin{array}{*8{c}}
0 & \pm 1 & 0 & 0 & \pm 1 & \pm1 & 0 & \pm 1 \\
0 & 0 & \pm 1 & 0 & \pm 1 & 0 & \pm 1 & \pm 1 \\
0 & 0 & 0 & \pm 1 & 0 & \pm 1 & \pm 1 & \pm 1
\end{array}\!\right],
\]
for any choice of signs.

The combinatorial proof of $\h(1,m) = m^2+m+1$ is based on the fact that
\begin{align}
\textrm{No Sauer Matrix of size } 3 \textrm{ is totally } 1\textrm{-submodular.}\label{eqn:no-TU-Sauer-Matrix-size-3}
\end{align}
This is discussed in Schrijver~\cite[\S 21.3]{schrijver1986theory}, Bixby \& Cunningham~\cite{bixbycunningham1987short}, and Tutte~\cite{tutte1958ahomotopy}, and also implicitly in the analysis of the first equation on page 1361 of Heller's paper~\cite{heller1957onlinear}.
In order to extend this kind of argument to the (refinded) shifted setting, we need some more notation.

\begin{definition}
\label{def:feasible-sauer-matrices}
Let $S$ be a Sauer Matrix of size~$k$.
We say that a vector $r \in [0,1)^k$ is \emph{feasible for~$S$} if $r+S$ is totally $1$-submodular.
Further, we say that~$S$ is \emph{feasible for translations} if there exists a vector $r \in [0,1)^k$ that is feasible for~$S$, and otherwise we say that~$S$ is \emph{infeasible for translations}.

Moreover, the Sauer Matrix~$S$ is said to be \emph{of type $(s,k-s)$}, if there are exactly~$s$ rows in~$S$ that contain at least one entry equal to~$1$.
\end{definition}
Note that there is (up to permuting rows or columns) only one Sauer Matrix of type~$(0,k)$.
As feasibility of a Sauer Matrix of type $(s,k-s)$ is invariant under permuting rows, we always make the assumption that each of its first~$s$ rows contains an entry equal to~$1$.

\begin{proposition}
\label{prop:sauer-matrix-condition}
Let $m,k,\Delta \in \Z_{>0}$ be such that $m > k$.
\begin{enumerate}[label=(\roman*)]
 \item If no Sauer Matrix of size~$k$ is feasible for translations, then
\[
\hs(m) \leq 2 \cdot \sum_{i=0}^{k-1} \binom{m}{i} \in \cO(m^{k-1}).
\]
 
 \item  If no Sauer Matrix~$S$ of size~$k$ admits a translation vector $t \in [0,1)^k \cap (\frac{1}{\Delta} \Z)^k \setminus \{\zero\}$ such that $t+S$ is totally $1$-submodular, then
\[
\hsd{\Delta}{m} \leq 2 \cdot \sum_{i=0}^{k-1} \binom{m}{i} \in \cO(m^{k-1}).
\]
\end{enumerate}
\end{proposition}

\begin{proof}
(i): Assume for contradiction that there is a matrix $A \in \{-1,0,1\}^{m \times n}$ and a translation vector $t \in [0,1)^m$ such that $t+A$ is totally $1$-submodular and $n > 2 \sum_{i=0}^{k-1} \binom{m}{i}$.
By Proposition~\ref{prop:shifted-heller-supports}, we have $|\cE_A| \geq \frac12 n > \sum_{i=0}^{k-1} \binom{m}{i}$ and thus by Sauer's Lemma (up to permuting rows or columns) the matrix~$A$ has a Sauer Matrix~$S$ of size~$k$ as a submatrix.
Writing $r \in [0,1)^k$ for the restriction of~$t$ to the~$k$ rows of~$A$ in which we find the Sauer Matrix~$S$, we get that by the total $1$-submodularity of $t+A$, the matrix $r+S$ necessarily must be totally $1$-submodular as well.
This however contradicts the assumption.

(ii): The argument is analogous to the one given for part (i).
\end{proof}

In contrast to the unshifted setting, for the sizes~$3$ and~$4$, there are Sauer Matrices~$S$ and vectors~$r$, such that $r+S$ is totally $1$-submodular.
Consider, for instance, for size~$3$ the pair
\[
S =
{\small
\left[\!\begin{array}{*{8}r}
0 & -1 &  0 &  0 & -1 & -1 &  0 & -1 \\
0 &  0 & -1 &  0 & -1 &  0 & -1 & -1 \\
0 &  0 &  0 & -1 &  0 & -1 & -1 & -1
\end{array}\!\right]}
\,,\,
r = {\small
\begin{bmatrix}
1/2 \\
1/2 \\
1/2
\end{bmatrix}}
,
\]
and, for size~$4$ the pair
\[
S =
{\tiny
\left[\!\begin{array}{*{16}r}
0 & -1 &  0 &  0 &  0 & -1 & -1 & -1 &  0 &  0 &  0 & -1 & -1 & -1 &  0 & -1 \\
0 &  0 & -1 &  0 &  0 & -1 &  0 &  0 & -1 & -1 &  0 & -1 & -1 &  0 & -1 & -1 \\
0 &  0 &  0 & -1 &  0 &  0 & -1 &  0 & -1 &  0 & -1 & -1 &  0 & -1 & -1 & -1 \\
0 &  0 &  0 &  0 & -1 &  0 &  0 & -1 &  0 & -1 & -1 &  0 & -1 & -1 & -1 & -1
\end{array}\!\right]}
\,,\,
\]
and $r = (\frac12,\frac12,\frac12,\frac12)^\intercal$.
In both cases, $2(r+S)$ is a matrix all of whose entries are either~$1$ or~$-1$.
By Hadamard's inequality, the determinant of any $\pm 1$-matrix of size~$k \leq 4$ is at most~$2^k$, and thus $\Delta_k(r+S) \leq 1$ for all $k \leq 4$, in the two examples above.

Our aim is to show that this pattern does not extend to higher dimensions.
In particular, we prove that no Sauer Matrix of size~$5$ is feasible for translations, which leads to \Cref{thm:polynomial-linear}~(\romannumeral1), and that feasible vectors for Sauer Matrices of size~$4$ must have all its entries in $\{0,\frac12\}$ leading to \Cref{thm:polynomial-linear}~(\romannumeral2).
The proof requires a more detailed study of Sauer Matrices of special types and sizes~$4$ and~$5$.

\begin{proposition}
\label{prop:Sauer-size-4}
Let $S$ be a Sauer Matrix of size~$4$ and let $r \in [0,1)^4$ be feasible for~$S$.
\begin{enumerate}[label=(\roman*)]
 \item If $S$ is of type~$(0,4)$, then $r = (\frac12,\frac12,\frac12,\frac12)^\intercal$.
 \item If $S$ is of type~$(1,3)$, then $r = (0,\frac12,\frac12,\frac12)^\intercal$.
 \item If $S$ is of type~$(2,2)$, then $r = (0,0,\frac12,\frac12)^\intercal$.
 \item If $S$ is of type~$(3,1)$ or~$(4,0)$, then~$S$ is infeasible for translations.
\end{enumerate}
\end{proposition}

\begin{proposition}
\label{prop:sauer-matrix-infeasibility-size-5}
There does not exist a Sauer Matrix $S$ of size~$5$ and a translation vector $r \in [0,1)^5$ such that $r+S$ is totally $1$-submodular.
\end{proposition}

\noindent The proof of these statements is based on identifying certain full-rank submatrices of the respective Sauer Matrix for which the minor condition provides a strong obstruction for feasibility.
The technical details are given in \Cref{sec:sauer-matrix-proof}.

With these preparations we are now able to prove our main result.

\begin{proof}[{\Cref{thm:polynomial-linear}}]
(i): In view of \Cref{lem:residue-class-approach}, we have $\h(\Delta,m) \leq \h(1,m) + (\Delta - 1) \cdot \hs(m)$.
The claimed bound now follows by Heller's identity $\h(1,m) = m^2 + m + 1$ and the fact that $\hs(m) \leq 2 \sum_{i=0}^4 \binom{m}{i}$, which holds by combining \Cref{prop:sauer-matrix-condition}~(\romannumeral1) and \Cref{prop:sauer-matrix-infeasibility-size-5}.

(ii): Again, in view of \Cref{lem:residue-class-approach}, we have $\h(\Delta,m) \leq \h(1,m) + (\Delta - 1) \cdot \hsd{\Delta}{m}$.
By \Cref{prop:Sauer-size-4}, for every Sauer Matrix~$S$ of size~$4$, any vector $t \in [0,1)^4$ such that $t+S$ is totally $1$-submodular has all its coordinates in~$\{0,\frac12\}$.
If $\Delta$ is odd, then using \Cref{prop:coefficients}, such translation vectors are excluded from the definition of $\hsd{\Delta}{m}$.
Thus, by \Cref{prop:sauer-matrix-condition}~(\romannumeral2), we get $\hsd{\Delta}{m} \leq 2 \sum_{i=0}^3 \binom{m}{i}$, which together with Heller's identity $\h(1,m) = m^2 + m + 1$ proves the claimed bound.
\end{proof}

\section{Feasibility of Sauer Matrices in Low Dimensions}
\label{sec:sauer-matrix-proof}

Here, we complete the discussion from the previous section and give the technical details and the proof of \Cref{prop:Sauer-size-4,prop:sauer-matrix-infeasibility-size-5}.
Parts of the arguments are based on the fact that the condition $\card{\det(r+M)} \leq 1$, for any $M \in \R^{k \times k}$, is equivalent to a pair of linear inequalities in the coordinates of~$r \in \R^k$.
This turns the question on whether a given Sauer Matrix is feasible for translations into the question of whether an associated polyhedron is non-empty.
To this end, let $S \in \{-1,0,1\}^{k \times 2^k}$ be a Sauer Matrix of size~$k$ and consider the set
\[
\cP(S) = \left\{ r \in [0,1]^k : \Delta_\ell(r+S) \leq 1 \textrm{ for each } 1 \leq \ell \leq k \right\}
\]
of feasible vectors for~$S$ in the unit cube $[0,1]^k$.

\begin{proposition}
\label{prop:polyhedrality-feasible-set}
For every Sauer Matrix~$S$, the set $\cP(S)$ is a polytope.
\end{proposition}
\begin{proof}
A vector $r \in [0,1]^k$ is contained in~$\cP(S)$ if and only if $\card{\det(r_I+S_{I,J})} \leq 1$, for every $I \subseteq [k]$ and $J \subseteq [2^k]$ with $\card{I} = \card{J} = \ell$ and $1 \leq \ell \leq k$, and where~$r_I$ denotes the subvector of~$r$ with coordinates indexed by~$I$ and $S_{I,J}$ denotes the submatrix of~$S$ with rows and columns indexed by~$I$ and~$J$, respectively.
Now, in general, for an $\ell \times \ell$ matrix~$A$ with columns $a_1,\ldots,a_\ell$ and a vector $t \in \R^\ell$, we have
\begin{align}
\det(t+A) = \det(t + a_1,\ldots,t + a_\ell)
= \det\left(
\begin{matrix}
a_1 & \ldots & a_\ell & -t \\
  1 & \ldots &      1 &  1
\end{matrix}
\right).\label{eqn:linearity-det-A-t}
\end{align}
Therefore, the multilinearity of the determinant translates the condition $\card{\det(r_I+S_{I,J})} \leq 1$ into a pair of linear inequalities in the entries of~$r$.
\end{proof}

\subsection{Sauer Matrices of size 3}

We start by illustrating the polyhedrality of $\cP(S)$ on the Sauer Matrix $S_{(0,3)}$ of type $(0,3)$.
Note that the columns of $-S_{(0,3)}$ are given by the eight~$0/1$ vectors $\{0,1\}^3$.
Further, for any $3 \times 3$ submatrix $A = (a_1,a_2,a_3)$ of $S_{(0,3)}$, we have
\[
\card{\det(t+A)} = \card{\det(-a_1 - t,-a_2 - t,-a_3 - t)} = 6 \vol(\conv\{-a_1,-a_2,-a_3,t\}),
\]
so the condition $\card{\det(t+A)} \leq 1$ is a condition on the volume of the simplex in~$[0,1]^3$ with vertices $-a_1,-a_2,-a_3,t$.
It is straightforward to see that the only case in which this imposes a condition on~$t$ is when the vertices $-a_1,-a_2,-a_3 \in \{0,1\}^3$ are chosen such that they are pairwise not connected by an edge of the cube.
Also the $2 \times 2$ minors of $t+S_{(0,3)}$ do not further restrict the feasibility of~$t$ and in summary, we thus get
\begin{align}
\cP(S_{(0,3)}) &= \Big\{ r \in [0,1]^3 : 1 \leq r_1 + r_2 + r_3 \leq 2 , \ 0 \leq -r_1 + r_2 + r_3 \leq 1 , \nonumber \\
 &\phantom{= \Big\{ r \in [0,1]^3 :\ }\,0 \leq r_1 - r_2 + r_3 \leq 1 , \ 0 \leq \phantom{-}r_1 + r_2 - r_3 \leq 1 \,\Big\} \nonumber\\ 
&= \conv\left\{
{\small\begin{bmatrix}
0 \\ 1/2 \\ 1/2
\end{bmatrix}},
{\small\begin{bmatrix}
1/2 \\ 0 \\ 1/2
\end{bmatrix}},
{\small\begin{bmatrix}
1/2 \\ 1/2 \\ 0
\end{bmatrix}},
{\small\begin{bmatrix}
1 \\ 1/2 \\ 1/2
\end{bmatrix}},
{\small\begin{bmatrix}
1/2 \\ 1 \\ 1/2
\end{bmatrix}},
{\small\begin{bmatrix}
1/2 \\ 1/2 \\ 1
\end{bmatrix}}\right\}. \label{eqn:P-Sauer-type-03}
\end{align}
This is a regular crosspolytope with each of its vertices being the center of a facet of the cube~$[0,1]^3$.

Our second goal is to characterize the feasible vectors for Sauer Matrices of size~$3$ and type~$(1,2)$.

\begin{proposition}
\label{prop:Sauer-size-3-type-1-2}
Let $S$ be a Sauer Matrix of type~$(1,2)$, that is,
\[
S =
\left[\!\begin{array}{*8{c}}
0 & \pm 1 & 0 & 0 & \pm 1 & \pm1 & 0 & \pm 1 \\
0 & 0 & - 1 & 0 & - 1 & 0 & - 1 & - 1 \\
0 & 0 & 0 & - 1 & 0 & - 1 & - 1 & - 1
\end{array}\!\right]
\]
for any signs and such that at least one entry in the first row of~$S$ equals~$1$.
Let $\sigma \in \{+,-\}^4$ be the sign vector of the non-zero entries in the first row of~$S$, and let $r \in [0,1)^3$ be feasible for~$S$.
\begin{enumerate}[label=(\roman*)]
 \item If $\sigma \in \left\{ (+,-,+,-) , (-,+,-,+) \right\}$, then $r \in \conv\left\{ (0,\frac14,\frac12)^\intercal , (0,\frac34,\frac12)^\intercal \right\}$.
 \item If $\sigma \in \left\{ (+,+,-,-) , (-,-,+,+) \right\}$, then $r \in \conv\left\{ (0,\frac12,\frac14)^\intercal , (0,\frac12,\frac34)^\intercal \right\}$.
 \item If $\sigma$ is different from any of those in parts~(i) and~(ii), then $r = (0,\frac12,\frac12)^\intercal$.
\end{enumerate}
\end{proposition}

\begin{proof}
We start with some general observations:
By assumption, we have $r_1 = 0$ in any case, because there is an entry equal to~$1$ in the first row of~$S$.
This means that the possibilities for feasible vectors~$r$ are not affected by possibly negating the first row of~$S$.
It thus suffices to consider only those~$\sigma$ with $\sigma_1 = +$, and we make this assumption throughout the following.

Now, for any $i,j \in \{1,2,3\}$ and any $k,\ell \in \{1,2,\ldots,8\}$, we denote by $S_{ij,k\ell}$ the $2 \times 2$ submatrix of~$S$ with entries in rows either~$i$ or~$j$ and columns either~$k$ or~$\ell$.
With this notation and the condition $\Delta_2(r+S) \leq 1$ on the minors of size~$2$, we obtain:
\begin{enumerate}[label=(C12)]

 \item[(C$_{12}$)] If $\sigma_2 \neq \sigma_1$, then $r+S$ contains ${\small\begin{bmatrix}
0 \\ r_3
\end{bmatrix}}+S_{13,25} = \left[\!\begin{array}{*2{c}}
\pm 1   & \mp 1  \\
r_3 & r_3
\end{array}\!\right]$, which implies $r_3 \leq \frac12$.

 \item[(C$_{13}$)] If $\sigma_3 \neq \sigma_1$, then $r+S$ contains ${\small\begin{bmatrix}
0 \\ r_2
\end{bmatrix}}+S_{12,26} = \left[\!\begin{array}{*2{c}}
\pm 1   & \mp 1  \\
r_2 & r_2
\end{array}\!\right]$, which implies $r_2 \leq \frac12$.

 \item[(C$_{24}$)] If $\sigma_2 \neq \sigma_4$, then $r+S$ contains ${\small\begin{bmatrix}
0 \\ r_2
\end{bmatrix}}+S_{12,58} = \left[\!\begin{array}{*2{c}}
\pm 1   & \mp 1  \\
r_2 - 1 & r_2 - 1
\end{array}\!\right]$, which implies $r_2 \geq \frac12$.

 \item[(C$_{34}$)] If $\sigma_3 \neq \sigma_4$, then $r+S$ contains ${\small\begin{bmatrix}
0 \\ r_3
\end{bmatrix}}+S_{13,68} = \left[\!\begin{array}{*2{c}}
\pm 1   & \mp 1  \\
r_3 - 1 & r_3 - 1
\end{array}\!\right]$, which implies $r_3 \geq \frac12$.

\end{enumerate}
Finally, for any indices $i,j,k \in \{1,2,\ldots,8\}$, we denote by $S_{ijk}$ the $3 \times 3$ submatrix of~$S$ consisting of its columns indexed by~$i,j$, and~$k$.

{\bfseries (i)}: Let $\sigma = (+,-,+,-)$.
As the first two and the last two entries of~$\sigma$ differ, in view of the conditions (C$_{12}$) and (C$_{34}$), we get $r_3 = \frac12$.
Since $\sigma_1 = +$ and $\sigma_2 = -$, the matrix $r+S$ contains $r+S_{245} = \left[\!\begin{array}{*3{c}}
1   & 0       & -1  \\
r_2 & r_2     & r_2 - 1 \\
r_3 & r_3 - 1 & r_3
\end{array}\!\right]$,
for which the condition $\Delta_3(t+S) \leq 1$ yields $2r_2 + r_3 \leq 2$, and thus $r_2 \leq \frac34$.
Similarly, $r+S$ also contains $r+S_{257} = \left[\!\begin{array}{*3{c}}
1   & -1      & 0  \\
r_2 & r_2 - 1 & r_2 - 1 \\
r_3 & r_3     & r_3 - 1
\end{array}\!\right]$,
for which the condition $\Delta_3(t+S) \leq 1$ yields $2r_2 - r_3 \geq 0$, and thus $r_2 \geq \frac14$.

{\bfseries (ii)}: Since interchanging the last two rows of~$S$ means to interchange $\sigma_2$ and $\sigma_3$, and since this change of rows in~$S$ translates into exchanging the coordinates~$r_2$ and~$r_3$, the claim follows by~(i).

{\bfseries (iii)}: By negating the first row of~$S$, the case $\sigma = (+,+,+,+)$ corresponds to the Sauer Matrix $S_{(0,3)}$ of type~$(0,3)$ with the additional restriction that $r_1 = 0$.
In view of the characterization~\eqref{eqn:P-Sauer-type-03} of its feasible vectors, this leaves as the only possibility $r = (0,\frac12,\frac12)^\intercal$, as claimed.
So, we may assume that at least one of the coordinates $\sigma_2,\sigma_3,\sigma_4$ equals~$-$.

First, assume that $\sigma_1 = \sigma_4 = +$.
If $\sigma_2 = -$, then the submatrix $S_{248}$ of~$S$ yields $r_2 \leq r_3$ via the minor condition $\Delta_3(r+S)$ just as in part~(i).
Since by (C$_{12}$) and (C$_{24}$) we have $r_3 \leq \frac12$ and $r_2 \geq \frac12$, respectively, we obtain $r = (0,\frac12,\frac12)^\intercal$, as claimed.
An analogous argument works for the case that $\sigma_3 = -$, by using the submatrix~$S_{238}$ of~$S$.

Second, assume that $\sigma_4 = - \neq \sigma_1$.
The only cases left to consider are those for which $\sigma_2 = \sigma_3$.
If $\sigma_2 = \sigma_3 = +$, then by (C$_{24}$) and (C$_{34}$), we get $r_2 \geq \frac12$ and $r_3 \geq \frac12$, respectively.
Moreover, using the submatrix~$S_{156}$ of~$S$ as before yields that $r_2 + r_3 \leq 1$ implying the desired $r = (0,\frac12,\frac12)^\intercal$.
Again, the case $\sigma_2 = \sigma_3 = -$ is completely analogous, using the conditions (C$_{12}$) and~(C$_{13}$), and the submatrix~$S_{567}$ of~$S$.
\end{proof}

\subsection{Sauer Matrices of size 4}

Based on the knowledge of feasibility of Sauer Matrices of size~$3$, we are now in position to prove \Cref{prop:Sauer-size-4}.
Sauer Matrices of the types $(0,4)$, $(3,1)$ and $(4,0)$ are easy to deal with, so let us start with those.

\begin{proof}[Proof of \Cref{prop:Sauer-size-4}~(i) and~(iv)]
{\bfseries (i)}: Assume that $r \in [0,1)^4$ is such that $r+S$ is totally $1$-submodular, and consider the following two $4 \times 4$ submatrices of~$S$:
\begin{align}
M = {\tiny
\left[\!\begin{array}{r|rrr}
 0 &  0 &  0 &  0 \\\hline
 0 &  0 & -1 & -1 \\
 0 & -1 &  0 & -1 \\
 0 & -1 & -1 &  0 
\end{array}\!\right]}
\quad\textrm{ and }\quad
N = {\tiny
\left[\!\begin{array}{r|rrr}
-1 & -1 & -1 & -1 \\\hline
-1 & -1 &  0 &  0 \\
-1 &  0 & -1 &  0 \\
-1 &  0 &  0 & -1 
\end{array}\!\right]}.\label{eqn:matrix-N}
\end{align}
By the $4 \times 4$ minor condition on $r+S$, we have
\[
\card{\det(r+M)} = r_1 \cdot \det
{\tiny
\begin{bmatrix}
0 & 1 & 1 \\
1 & 0 & 1 \\
1 & 1 & 0 
\end{bmatrix}} = 2 r_1 \leq 1,
\]
and hence $r_1 \leq \frac12$.
Likewise, we have
\[
\card{\det(r+N)} = (1-r_1) \cdot \det
{\tiny
\begin{bmatrix}
0 & 1 & 1 \\
1 & 0 & 1 \\
1 & 1 & 0 
\end{bmatrix}} = 2 (1-r_1) \leq 1,
\]
and hence $r_1 \geq \frac12$, so that actually $r_1 = \frac12$.
Analogous arguments for the other coordinates of~$r$, show that $r = (\frac12,\frac12,\frac12,\frac12)^\intercal$ as claimed.

{\bfseries (iv)}: If in the $i$th row of a Sauer Matrix~$S$ there is an entry equal to~$1$, then $r_i=0$, because of $\Delta_1(r+S) \leq 1$.
In the types $(3,1)$ and $(4,0)$, the first three rows in~$S$ contain an entry equal to~$1$, so that they contains a Sauer Matrix of size~$3$ that is itself totally $1$-submodular.
However, we noted in~\eqref{eqn:no-TU-Sauer-Matrix-size-3} that no such Sauer Matrix exists.
\end{proof}

\begin{proof}[{Proof of \Cref{prop:Sauer-size-4}~(ii)}]
We assume that the first row of each considered Sauer Matrix~$S$ of type $(1,3)$ contains an entry equal to~$1$, so that $r_1 = 0$.
We can moreover assume that $(-1,-1,-1,-1)^\intercal$ is a column of~$S$ (by possibly multiplying the first row by~$-1$).
We now employ a case distinction based on the signs of the entries in the first row of the columns $a = (\pm 1,-1,0,0)^\intercal$, $b = (\pm1,0,-1,0)^\intercal$, and $c = (\pm1,0,0,-1)^\intercal$ of~$S$.

\smallskip
\noindent\emph{Case 1:} $a_1 = b_1 = c_1 = -1$.

Under this assumption, $S$ contains the matrix~$N$ from~\eqref{eqn:matrix-N} as a submatrix and thus $r_1 \geq \frac12$, contradicting that $r_1 = 0$.

\smallskip
\noindent\emph{Case 2:} $a_1 = b_1 = c_1 = 1$.

 In this case, $S$ contains the submatrices
\[
A = {\tiny
\left[\!\begin{array}{r|rrr}
 0 &  1 &  1 &  1 \\\hline
 0 & -1 &  0 &  0 \\
 0 &  0 & -1 &  0 \\
 0 &  0 &  0 & -1 
\end{array}\!\right]}
\quad\textrm{ and }\quad
B = {\tiny
\left[\!\begin{array}{r|rrr}
 0 &  1 &  1 &  1 \\\hline
-1 & -1 &  0 &  0 \\
-1 &  0 & -1 &  0 \\
-1 &  0 &  0 & -1 
\end{array}\!\right]}.
\]
The conditions $\card{\det(r+A)} \leq 1$ and $\card{\det(r+B)} \leq 1$ translate into the contradicting inequalities $r_2+r_3+r_4 \leq 1$ and $r_2+r_3+r_4 \geq 2$, respectively.

\smallskip
\noindent\emph{Case 3:} Exactly two of the entries $a_1,b_1,c_1$ equal~$-1$.

Without loss of generality, we may permute the last three rows of~$S$, and assume that $a_1 = b_1 = -1$.
 We find that $S$ now contains the submatrices
\[
C = {\tiny
\left[\!\begin{array}{r|rrr}
 0 & -1 & -1 &  0 \\\hline
 0 & -1 &  0 &  0 \\
 0 &  0 & -1 &  0 \\
 0 &  0 &  0 & -1 
\end{array}\!\right]}
\,,\,
D = {\tiny
\left[\!\begin{array}{r|rrr}
 0 & -1 & -1 &  0 \\\hline
-1 & -1 &  0 &  0 \\
-1 &  0 & -1 &  0 \\
-1 &  0 &  0 & -1 
\end{array}\!\right]}
\,,\,
E = {\tiny
\left[\!\begin{array}{r|rrr}
-1 & -1 & -1 &  0 \\\hline
-1 & -1 &  0 & -1 \\
-1 &  0 & -1 & -1 \\
-1 &  0 &  0 &  0 
\end{array}\!\right]}.
\]
The conditions $\card{\det(r+C)} \leq 1$, $\card{\det(r+D)} \leq 1$ and $\card{\det(r+E)} \leq 1$ translate into the contradicting inequalities $r_2 + r_3 \leq 1$, $r_4 \geq \frac12$, and $r_4 + 1 \leq r_2 + r_3$, respectively.

\smallskip
\noindent\emph{Case 4:} Exactly two of the entries $a_1,b_1,c_1$ equal~$1$.

As in Case~3, we may assume that $a_1 = b_1 = 1$.
Here, the following six matrices can be found as submatrices in~$S$:
\[
\tiny
\left[\!\begin{array}{r|rrr}
-1 &  0 &  0 & -1 \\\hline
-1 & -1 &  0 &  0 \\
-1 &  0 & -1 &  0 \\
-1 &  0 &  0 & -1 
\end{array}\!\right]
\ ,\ 
\left[\!\begin{array}{r|rrr}
-1 &  0 &  0 & -1 \\\hline
-1 & -1 & -1 &  0 \\
-1 &  0 &  0 &  0 \\
-1 &  0 & -1 & -1 
\end{array}\!\right]
\ ,\ 
\left[\!\begin{array}{r|rrr}
-1 &  0 &  0 & -1 \\\hline
-1 &  0 &  0 &  0 \\
-1 & -1 & -1 &  0 \\
-1 & -1 &  0 & -1 
\end{array}\!\right],
\]
\[
\tiny
\left[\!\begin{array}{r|rrr}
 0 &  1 &  1 &  0 \\\hline
-1 & -1 &  0 &  0 \\
-1 &  0 & -1 &  0 \\
-1 &  0 &  0 & -1 
\end{array}\!\right]
\ ,\ 
\left[\!\begin{array}{r|rrr}
 0 &  1 &  1 &  0 \\\hline
-1 & -1 &  0 & -1 \\
-1 &  0 & -1 & -1 \\
-1 &  0 &  0 &  0 
\end{array}\!\right]
\ ,\ 
\left[\!\begin{array}{r|rrr}
 0 &  1 &  1 &  0 \\\hline
 0 & -1 &  0 &  0 \\
 0 &  0 & -1 &  0 \\
 0 &  0 &  0 & -1 
\end{array}\!\right].
\]
The minor conditions for these matrices translate into the inequality system
\begin{align*}
r_4 &\leq \tfrac12 & r_3 &\leq r_2 & r_2 &\leq r_3 \\
r_4 &\geq \tfrac12 & r_2 + r_3 &\geq 1 & r_2 + r_3 &\leq 1
\end{align*}
in the same order as the matrices were given above.
Solving this system of inequalities shows that necessarily $r_2 = r_3 = r_4 = \frac12$, and the proof is complete.
\end{proof}

\begin{proof}[{Proof of \Cref{prop:Sauer-size-4}~(iii)}]
We write a Sauer Matrix~$S$ of type $(2,2)$ in the form
\[
{\small
\left[\!\begin{array}{rrrr|rrrr|rrrr|rrrr}
\multicolumn{4}{c|}{$Block 1$} & \multicolumn{4}{c|}{$Block 2$} & \multicolumn{4}{c|}{$Block 3$} & \multicolumn{4}{c}{$Block 4$} \\ \hline
0 & 0 & 0 & 0 &  \pm1 & \pm1 & \pm1 & \pm1 &  0 &  0 &  0 & 0 & \pm1 & \pm1 & \pm1 & \pm1 \\ \hline
0 & 0 & 0 & 0 &  0 &  0 &  0 &  0 & \pm1 & \pm1 & \pm1 & \pm1 & \pm1 & \pm1 & \pm1 & \pm1 \\ \hline
0 &  -1 &  0 & -1 &  0 &  -1 &  0 & -1 & 0 &  -1 &  0 & -1 & 0 &  -1 &  0 & -1  \\
0 &  0 &  -1 &  -1 & 0 &  0 &  -1 &  -1 & 0 &  0 &  -1 &  -1 & 0 &  0 &  -1 &  -1 
\end{array}\!\right]}
\]
where in each of the first two rows we have at least one entry equal to~$1$.
We find four Sauer Matrices of type $(1,2)$ as submatrices of~$S$:
\begin{align*}
M_{1,1,2} - \textrm{consisting of rows } 1,3,4 \textrm{ and the columns in Block 1 and Block 2}\\
M_{1,3,4} - \textrm{consisting of rows } 1,3,4 \textrm{ and the columns in Block 3 and Block 4}\\
M_{2,1,4} - \textrm{consisting of rows } 2,3,4 \textrm{ and the columns in Block 1 and Block 4}\\
M_{2,2,3} - \textrm{consisting of rows } 2,3,4 \textrm{ and the columns in Block 2 and Block 3}
\end{align*}
Now, let $r \in [0,1)^4$ be feasible for~$S$.
As the first two rows of~$S$ both contain an entry equal to~$1$, we always have $r_1 = r_2 = 0$.
We find that $(0,r_3,r_4)^\intercal$ must be a feasible vector for the Sauer Matrices $M_{1,1,2}$, $M_{1,3,4}$, $M_{2,1,4}$, and $M_{2,2,3}$.
Using \Cref{prop:Sauer-size-3-type-1-2}, we see that if we want to allow the possibility of $r$ being different from $(0,0,\frac12,\frac12)^\intercal$, then the sign patterns in the $1 \times 4$ blocks $[\pm1,\pm1,\pm1,\pm1]$ in those four matrices must all belong either to $\{(+,-,+,-) , (-,+,-,+)\}$ or to $\{(+,+,-,-) , (-,-,+,+)\}$.
Moreover, as $r_1 = r_2 = 0$, we may multiply the first or the second row of~$S$ with $-1$ without loosing the total $1$-submodularity of $r+S$.
This means that we may assume that the sign patterns in the $1 \times 4$ blocks corresponding to $(\textrm{Row 1},\textrm{Block 2})$ and $(\textrm{Row 2},\textrm{Block 3})$ are the same.
Again as $r_1 = r_2 = 0$, total $1$-submodularity of $r+S$ is also not affected if we exchange the first two rows of~$S$, or Block 2 with Block 3.

These reductions show that the only Sauer Matrices of type $(2,2)$ that possibly allow feasible translation vectors $r = (0,0,r_3,r_4)^\intercal$ different from $(0,0,\frac12,\frac12)^\intercal$ have its first two rows given by
\begin{align*}
R_1 &= \left[\!\begin{array}{rrrr|rrrr|rrrr|rrrr}
0 & 0 & 0 & 0 & 1 & -1 & 1 & -1 &  0 &  0 &  0 & 0 & 1 & -1 & 1 & -1 \\
0 & 0 & 0 & 0 &  0 &  0 &  0 &  0 & 1 & -1 & 1 & -1 & 1 & -1 & 1 & -1
\end{array}\!\right], \\
R_2 &= \left[\!\begin{array}{rrrr|rrrr|rrrr|rrrr}
0 & 0 & 0 & 0 & 1 & -1 & 1 & -1 &  0 &  0 &  0 & 0 & 1 & -1 & 1 & -1 \\
0 & 0 & 0 & 0 &  0 &  0 &  0 &  0 & 1 & -1 & 1 & -1 & -1 & 1 & -1 & 1
\end{array}\!\right], \\
R_3 &= \left[\!\begin{array}{rrrr|rrrr|rrrr|rrrr}
0 & 0 & 0 & 0 & 1 & -1 & 1 & -1 &  0 &  0 &  0 & 0 & -1 & 1 & -1 & 1 \\
0 & 0 & 0 & 0 &  0 &  0 &  0 &  0 & 1 & -1 & 1 & -1 & -1 & 1 & -1 & 1
\end{array}\!\right],
\end{align*}
or by
\begin{align*}
R_4 &= \left[\!\begin{array}{rrrr|rrrr|rrrr|rrrr}
0 & 0 & 0 & 0 & 1 & 1 & -1 & -1 &  0 &  0 &  0 & 0 & 1 & 1 & -1 & -1 \\
0 & 0 & 0 & 0 &  0 &  0 &  0 &  0 & 1 & 1 & -1 & -1 & 1 & 1 & -1 & -1
\end{array}\!\right], \\
R_5 &= \left[\!\begin{array}{rrrr|rrrr|rrrr|rrrr}
0 & 0 & 0 & 0 & 1 & 1 & -1 & -1 &  0 &  0 &  0 & 0 & 1 & 1 & -1 & -1 \\
0 & 0 & 0 & 0 &  0 &  0 &  0 &  0 & 1 & 1 & -1 & -1 & -1 & -1 & 1 & 1
\end{array}\!\right], \\
R_6 &= \left[\!\begin{array}{rrrr|rrrr|rrrr|rrrr}
0 & 0 & 0 & 0 & 1 & 1 & -1 & -1 &  0 &  0 &  0 & 0 & -1 & -1 & 1 & 1 \\
0 & 0 & 0 & 0 &  0 &  0 &  0 &  0 & 1 & 1 & -1 & -1 & -1 & -1 & 1 & 1
\end{array}\!\right].
\end{align*}
Observe also that for the first three matrices, the feasible vectors are of the form $r = (0,0,r_3,\frac12)^\intercal$, for $\frac14 \leq r_3 \leq \frac34$, whereas for the second three matrices they have the form $r = (0,0,\frac12,r_4)^\intercal$, for $\frac14 \leq r_4 \leq \frac34$.
Let us denote the $4 \times 4$ submatrix of~$S$ indexed by the columns $i,j,k,\ell \in \{1,2,\ldots,16\}$ by $S_{i,j,k,\ell}$.

\smallskip
\noindent\emph{Case 1:} First two rows of~$S$ either $R_1$, $R_2$, or $R_3$, and $r  = (0,0,r_3,\frac12)^\intercal$.

For $R_1$, applying the condition $\Delta_4(r+S) \leq 1$ to the submatrices $S_{1,6,10,16}$ and $S_{2,5,9,15}$, yields $r_3 \leq \frac12$ and $r_3 \geq \frac12$, respectively.
Thus $r = (0,0,\frac12,\frac12)^\intercal$.

For $R_2$, applying the condition $\Delta_4(r+S) \leq 1$ to the submatrices $S_{1,5,10,15}$ and $S_{2,6,9,16}$, yields $r_3 \leq \frac12$ and $r_3 \geq \frac12$, respectively.
Thus $r = (0,0,\frac12,\frac12)^\intercal$.

For $R_3$, the submatrix $S_{1,6,10,14}$ gives determinant $\card{\det(r+S_{1,6,10,14})} = \frac32$, contradicting the condition $\Delta_4(r+S) \leq 1$.
Thus, there is no feasible vector $r \in [0,1)^4$ at all in this case.

\smallskip
\noindent\emph{Case 2:} First two rows of~$S$ either $R_4$, $R_5$, or $R_6$, and $r  = (0,0,\frac12,r_4)^\intercal$.

For $R_4$, applying the condition $\Delta_4(r+S) \leq 1$ to the submatrices $S_{1,6,10,13}$ and $S_{3,5,9,14}$, yields $r_3 \leq \frac12$ and $r_3 \geq \frac12$, respectively.
Thus $r = (0,0,\frac12,\frac12)^\intercal$.

For $R_5$, applying the condition $\Delta_4(r+S) \leq 1$ to the submatrices $S_{1,5,10,14}$ and $S_{3,6,9,13}$, yields $r_3 \leq \frac12$ and $r_3 \geq \frac12$, respectively.
Thus $r = (0,0,\frac12,\frac12)^\intercal$.

For $R_6$, the submatrix $S_{1,7,11,15}$ gives determinant $\card{\det(r+S_{1,6,10,14})} = \frac32$, contradicting the condition $\Delta_4(r+S) \leq 1$.
Thus, there is no feasible vector $r \in [0,1)^4$ at all in this case.

\smallskip
Summarizing our results from above, we see that if a Sauer Matrix of type $(2,2)$ admits a feasible vector $r \in [0,1)^4$, then $r = (0,0,\frac12,\frac12)^\intercal$, as desired.
\end{proof}

\subsection{Sauer Matrices of size 5}

\begin{proposition}\ 
\label{prop:negative-Sauer-size-5}
\begin{enumerate}[label=(\roman*)]
 \item The Sauer Matrix of type $(0,5)$ is infeasible for translations.
 \item No Sauer Matrix of type $(1,4)$ is feasible for translations.
 \item No Sauer Matrix of type $(2,3)$ is feasible for translations.
\end{enumerate}
\end{proposition}

\begin{proof}
{\bfseries (i)}: The argument is similar to the one for \Cref{prop:Sauer-size-4}~(\romannumeral1).
Assume for contradiction, that there is a vector $r \in [0,1)^5$ with $\Delta_5(r+S) \leq 1$.
Consider the following two $5 \times 5$ submatrices of~$S$:
\[
X = {\tiny
\left[\!\begin{array}{r|rrrr}
 0 &  0 &  0 &  0 &  0 \\\hline
 0 &  0 & -1 & -1 & -1 \\
 0 & -1 &  0 & -1 & -1 \\
 0 & -1 & -1 &  0 & -1 \\
 0 & -1 & -1 & -1 &  0
\end{array}\!\right]}
\quad\textrm{ and }\quad
Y = {\tiny
\left[\!\begin{array}{r|rrrr}
-1 & -1 & -1 & -1 & -1 \\\hline
-1 & -1 &  0 &  0 &  0 \\
-1 &  0 & -1 &  0 &  0 \\
-1 &  0 &  0 & -1 &  0 \\
-1 &  0 &  0 &  0 & -1
\end{array}\!\right]}.
\]
By the $5 \times 5$ minor condition on $r+S$, we have
\[
\card{\det(r+X)} = r_1 \cdot \det
{\tiny
\begin{bmatrix}
0 & 1 & 1 & 1 \\
1 & 0 & 1 & 1 \\
1 & 1 & 0 & 1 \\
1 & 1 & 1 & 0
\end{bmatrix}} = 3 r_1 \leq 1,
\]
and hence $r_1 \leq \frac13$.
Likewise, we have
\[
\card{\det(r+Y)} = (1-r_1) \cdot \det
{\tiny
\begin{bmatrix}
0 & 1 & 1 & 1 \\
1 & 0 & 1 & 1 \\
1 & 1 & 0 & 1 \\
1 & 1 & 1 & 0
\end{bmatrix}} = 3 (1-r_1) \leq 1.
\]
Therefore, we get $r_1 \geq \frac23$, a contradiction.

{\bfseries (ii)}: Let $S$ be a Sauer Matrix of type $(1,4)$ and without loss of generality, we may assume that the first row of~$S$ contains an entry equal to~$1$.
We also assume for contradiction that there is some $r \in [0,1)^5$ such that $r+S$ is totally $1$-submodular.
As the entries of $r+S$ are contained in $[-1,1]$, we get that $r_1 = 0$.
Moreover, the last four rows of~$S$ contain a Sauer Matrix of type $(0,4)$.
By \Cref{prop:Sauer-size-4}~(\romannumeral1), this means that $r_2 = r_3 = r_4 = r_5 = \frac12$, so that in summary there is only one possibility for the translation vector~$r$.

Now, as $r_1 = 0$, we may multiply the first row of~$S$ with $-1$ if needed, and can assume that the vector $(-1,-1,-1,-1,-1)^\intercal$ is a column of~$S$.
If~$M$ denotes any of the four matrices
\begin{align*}
&{\tiny
\left[\!\begin{array}{r|rrrr}
-1 & -1 &  0 &  0 &  0 \\\hline
-1 & -1 &  0 &  0 &  0 \\
-1 &  0 & -1 &  0 &  0 \\
-1 &  0 &  0 & -1 &  0 \\
-1 &  0 &  0 &  0 & -1
\end{array}\!\right]
\,,\,
\left[\!\begin{array}{r|rrrr}
-1 &  0 & -1 &  0 &  0 \\\hline
-1 & -1 &  0 &  0 &  0 \\
-1 &  0 & -1 &  0 &  0 \\
-1 &  0 &  0 & -1 &  0 \\
-1 &  0 &  0 &  0 & -1
\end{array}\!\right]
\,,} \\
&{\tiny
\left[\!\begin{array}{r|rrrr}
-1 &  0 &  0 & -1 &  0 \\\hline
-1 & -1 &  0 &  0 &  0 \\
-1 &  0 & -1 &  0 &  0 \\
-1 &  0 &  0 & -1 &  0 \\
-1 &  0 &  0 &  0 & -1
\end{array}\!\right]
\,,\,
\left[\!\begin{array}{r|rrrr}
-1 &  0 &  0 &  0 & -1 \\\hline
-1 & -1 &  0 &  0 &  0 \\
-1 &  0 & -1 &  0 &  0 \\
-1 &  0 &  0 & -1 &  0 \\
-1 &  0 &  0 &  0 & -1
\end{array}\!\right]\,,
}
\end{align*}
then the absolute value of the determinant of $r+M$ equals~$3/2$.
Thus, if indeed $\Delta_5(r+S) \leq 1$, then these matrices cannot be submatrices of~$S$.
In particular, this implies that
\[
M' =
{\tiny
\left[\!\begin{array}{r|rrrr}
 0 &  1 &  1 &  1 &  1 \\\hline
-1 & -1 &  0 &  0 &  0 \\
-1 &  0 & -1 &  0 &  0 \\
-1 &  0 &  0 & -1 &  0 \\
-1 &  0 &  0 &  0 & -1
\end{array}\!\right]}
\]
must be a submatrix of~$S$.
However, the determinant of $r+M'$ equals~$-2$, in contradiction to $r+S$ being totally $1$-submodular.

{\bfseries (iii)}: Assume that there is a Sauer Matrix~$S$ of type $(2,3)$ and a vector $r \in [0,1)^5$ that is feasible for~$S$.
Observe that~$S$ contains feasible Sauer Matrices of types $(1,3)$ in its rows indexed by~$\{1,3,4,5\}$ and by~$\{2,3,4,5\}$.
By \Cref{prop:Sauer-size-4}~(ii) this means that necessarily we have $r = (0,0,\frac12,\frac12,\frac12)^\intercal$, and we can now argue similarly as we did in part~(ii).

First of all, as $r_1 = r_2 = 0$, we may multiply the first or second row of~$S$ with $-1$ if needed, and can assume that the vectors $(-1,0,-1,-1,-1)^\intercal$ and $(0,-1,0,0,0)^\intercal$ are columns of~$S$.
We distinguish cases based on the signs of the entries in the first or second row of the columns $a = (\pm 1,0,-1,0,0)^\intercal$, $b = (\pm1,0,0,-1,0)^\intercal$, $c = (\pm1,0,0,0,-1)^\intercal$, and $a' = (0,\pm 1,-1,0,0)^\intercal$, $b' = (0,\pm1,0,-1,0)^\intercal$, $c' = (0,\pm1,0,0,-1)^\intercal$ of~$S$.

\smallskip
\noindent\emph{Case 1:} $a_1 = b_1 = c_1 = 1$ or $a'_2 = b'_2 = c'_2 = -1$.

Here, one of the matrices
\[
C_1 =
{\tiny
\left[\!\begin{array}{rr|rrr}
 0 &  0 &  1 &  1 &  1 \\
 0 & -1 &  0 &  0 &  0 \\\hline
-1 &  0 & -1 &  0 &  0 \\
-1 &  0 &  0 & -1 &  0 \\
-1 &  0 &  0 &  0 & -1
\end{array}\!\right]}
\quad \textrm{ or } \quad
C_2 =
{\tiny
\left[\!\begin{array}{rr|rrr}
 0 & -1 &  0 &  0 &  0 \\
 0 &  0 & -1 & -1 & -1 \\\hline
 0 & -1 & -1 &  0 &  0 \\
 0 & -1 &  0 & -1 &  0 \\
 0 & -1 &  0 &  0 & -1
\end{array}\!\right]}
\]
must be a submatrix of~$S$, but the absolute value of the determinant of both $r+C_1$ and $r+C_2$ equals $3/2$.

\smallskip
\noindent\emph{Case 2:} Two of the entries $a_1,b_1,c_1$ equal $-1$ or two of the entries $a'_2,b'_2,c'_2$ equal~$1$.

Without loss of generality, we may permute the last three rows of~$S$, and assume that either $a_1 = b_1 = -1$ or $a'_2 = b'_2 = 1$.
Now, one of the matrices
\[
C_3 =
{\tiny
\left[\!\begin{array}{rr|rrr}
-1 &  0 & -1 & -1 &  0 \\
 0 & -1 &  0 &  0 &  0 \\\hline
-1 &  0 & -1 &  0 &  0 \\
-1 &  0 &  0 & -1 &  0 \\
-1 &  0 &  0 &  0 & -1
\end{array}\!\right]}
\quad \textrm{ or } \quad
C_4 =
{\tiny
\left[\!\begin{array}{rr|rrr}
-1 &  0 &  0 &  0 &  0 \\
 0 & -1 &  1 &  1 &  0 \\\hline
-1 &  0 & -1 &  0 &  0 \\
-1 &  0 &  0 & -1 &  0 \\
-1 &  0 &  0 &  0 & -1
\end{array}\!\right]}
\]
must be a submatrix of~$S$, but again the absolute value of the determinant of both $r+C_3$ and $r+C_4$ equals $3/2$.

\smallskip
\noindent\emph{Case 3:} Up to permuting the last three rows of~$S$ we have
$\tiny
\left[\!\begin{array}{*{3}r}
 a_1 &  b_1 &  c_1 \\
 a'_2 & b'_2 & c'_2
\end{array}\!\right]
=
\left[\!\begin{array}{*{3}r}
-1 &  1 &  1 \\
 1 & -1 & -1
\end{array}\!\right]
$.

With this assumption, one of the matrices
\begin{align*}
&{\tiny
\left[\!\begin{array}{rr|rrr}
-1 &  0 &  0 &  1 &  1 \\
-1 & -1 &  0 &  0 &  0 \\\hline
-1 &  0 & -1 &  0 &  0 \\
-1 &  0 &  0 & -1 &  0 \\
-1 &  0 &  0 &  0 & -1
\end{array}\!\right]}
\ ,\ 
{\tiny
\left[\!\begin{array}{rr|rrr}
-1 & -1 &  0 &  0 &  0 \\
 1 &  0 &  1 &  0 &  0 \\\hline
-1 & -1 & -1 &  0 &  0 \\
-1 &  0 &  0 & -1 &  0 \\
-1 &  0 &  0 &  0 & -1
\end{array}\!\right]}
\ ,\ \\
&{\tiny
\left[\!\begin{array}{rr|rrr}
 1 &  0 & -1 &  1 &  1 \\
-1 & -1 &  0 &  0 &  0 \\\hline
-1 &  0 & -1 &  0 &  0 \\
-1 &  0 &  0 & -1 &  0 \\
-1 &  0 &  0 &  0 & -1
\end{array}\!\right]}
\ ,\ 
{\tiny
\left[\!\begin{array}{rr|rrr}
 1 & -1 &  0 &  0 &  0 \\
 1 &  0 &  0 & -1 & -1 \\\hline
-1 & -1 & -1 &  0 &  0 \\
-1 & -1 &  0 & -1 &  0 \\
-1 & -1 &  0 &  0 & -1
\end{array}\!\right]}
\end{align*}
must be a submatrix of~$S$, because one of the four vectors $(\pm1,\pm1,-1,-1,-1)^\intercal$ must be a column of~$S$.
As before, if~$F$ denotes any of these four matrices, then the absolute value of the determinant of~$r+F$ equals~$3/2$.

\smallskip
\noindent\emph{Case 4:} Up to permuting the last three rows of~$S$ we have
$\tiny
\begin{bmatrix}
 a_1 &  b_1 &  c_1 \\
a'_2 & b'_2 & c'_2
\end{bmatrix}
=
\left[\!\begin{array}{rrr}
-1 &  1 &  1 \\
-1 & -1 &  1
\end{array}\!\right]
$.

In this case, one of the matrices
\[
C_7 =
{\tiny
\left[\!\begin{array}{rr|rrr}
-1 &  0 &  1 &  0 &  0 \\
 0 & -1 &  0 & -1 &  0 \\\hline
 0 & -1 &  0 &  0 &  0 \\
 0 &  0 & -1 & -1 &  0 \\
 0 &  0 &  0 &  0 & -1
\end{array}\!\right]}
\quad \textrm{ or } \quad
C_8 =
{\tiny
\left[\!\begin{array}{rr|rrr}
 1 & -1 &  0 &  0 &  0 \\
 0 &  0 & -1 & -1 &  0 \\\hline
 0 & -1 & -1 &  0 &  0 \\
 0 &  0 &  0 & -1 &  0 \\
 0 &  0 &  0 &  0 & -1
\end{array}\!\right]}
\]
must be a submatrix of~$S$, because one of the vectors $(\pm1,0,0,0,0)^\intercal$ must be a column of~$S$.
As before, the absolute value of the determinant of both $r+C_7$ and $r+C_8$ equals~$3/2$.

In conclusion, in all cases we found a $5 \times 5$ minor of $r+S$ whose absolute value is greater than~$1$, and thus no feasible Sauer Matrix of type~$(2,3)$ can exist.
\end{proof}

With these preparations we can now exclude the existence of any Sauer Matrix of size~$5$ that is feasible for translations.

\begin{proof}[Proof of \Cref{prop:sauer-matrix-infeasibility-size-5}]
Just as in the proof of \Cref{prop:Sauer-size-4}~(iv), whenever there are at least three rows in a Sauer Matrix that contain an entry equal to~$1$, then it is infeasible for translations.

Thus, we may assume that~$S$ is a Sauer Matrix whose type is either $(0,5)$, $(1,4)$, or $(2,3)$.
We have just proven in \Cref{prop:negative-Sauer-size-5} however, that all such Sauer Matrices are infeasible for translations.
\end{proof}

\section{Computational Experiments and the Proof of Theorem~\ref{thm:counterexamples-Delta-4}}
\label{sect:computations}

In this part, we describe a computational approach to determine so far unknown values of $\h(\Delta,m)$ for small parameters $m,\Delta \in \Z_{>0}$.
The results of our computations led us to identify a family of counterexamples to \Cref{conj:lee-exact-value-hDm} that lie behind the lower bounds in \Cref{thm:counterexamples-Delta-4}.
Our approach is based on the \emph{sandwich factory classification scheme} described and utilized by Averkov, Borger \& Soprunov~\cite[Section 7.3]{averkovborgersoprunov2021classification}.

As we have done implicitly already in previous sections, we now explicitly work with sets of integer points in~$\Z^m$, rather than with $\Delta$-modular integer matrices with~$m$ rows and full rank.
To this end, for a point set $S \subseteq \Z^m$, we write
\[
\Delta(S) := \max\left\{ \card{\det(S')} : S' \subseteq S , \card{S'} = m \right\}
\]
for the maximum absolute value of the determinant of a matrix whose columns constitute an $m$-element subset of~$S$.
If $P \subseteq \R^m$ is a \emph{lattice polytope}, meaning that all its vertices belong to~$\Z^m$, then we write $\Delta(P) := \Delta(P \cap \Z^m)$.
Since the maximum determinant is attained by an $m$-element subset of the vertices of~$P$, we also have $\Delta(P) = \Delta(\{x \in \Z^m : x \textrm{ a vertex of } P\})$.
For the same reason, we have $\Delta(\conv\{S\}) = \Delta(S)$, for every full-dimensional set $S \subseteq \Z^m$.
Since the value of $\h(\Delta,m)$ is attained by a matrix~$A$ whose columns come in opposite pairs $A_i, -A_i$, we restrict our attention to \emph{$o$-symmetric} lattice polytopes $P \subseteq \R^m$, that is, we require $P = -P$ to hold.
In this language the generalized Heller constant now expresses as
\begin{align}
	\h(\Delta,m) = \max\bigl\{ \card{P \cap \Z^m} : P \subseteq \R^m &\textrm{ an }o\textrm{-symmetric lattice polytope} \bigr.  \nonumber\\
	\bigl. & \text{ with } \Delta(P) = \Delta \bigr\}.\label{eqn:hDm-via-polytopes}
\end{align}
\noindent Therefore, in order to computationally determine the value $\h(\Delta,m)$ for a pair $(\Delta,m)$ of parameters, we may want to solve any of the following classification problems.

\begin{prob}
\label{prob:classification-full}
Given $m,\Delta \in \Z_{>0}$, classify up to unimodular equivalence all $o$-symmetric lattice polytopes $P \subseteq \R^m$ with $\Delta(P) = \Delta$.
\end{prob}

The second classification problem is a variant of the first and is computationally less expensive.

\begin{prob}
\label{prob:classification-maximal}
Given $m,\Delta \in \Z_{>0}$, classify up to unimodular equivalence all $o$-symmetric lattice polytopes $Q \subseteq \R^m$ with $\Delta(Q) = \Delta$ and with the maximal number of integer points under these constraints.
\end{prob}

\subsection{Classification by Sandwich Factory Approach}

As hinted above, we tackle these problems with the sandwich factory classification scheme of Averkov, Borger \& Soprunov~\cite[Section 7.3]{averkovborgersoprunov2021classification}.
This is a quite general and versatile approach that can be applied to various enumeration problems for lattice polytopes.

The basic idea is to use a so-called \emph{sandwich} $(A,B)$ with the inner part~$A$ and the outer part $B$, that is, a pair of lattice polytopes satisfying the inclusion $A \subseteq B$.
Such a sandwich represents the family of all lattice polytopes $P$ that are unimodularly equivalent to a lattice polytope $P'$ satisfying $A \subseteq P' \subseteq B$.
If the latter condition holds, one says that $P$ occurs in the sandwich $(A,B)$. 
If a family $\cF$ of polytopes needs to be enumerated up to unimodular equivalence, and it is possible to find a finite list of sandwiches with the property that every polytope $P \in \cF$ occurs in one of the sandwiches of the finite list, then the enumeration is carried out by iteratively refining sandwiches with $A \varsubsetneq B$, for which the discrepancy between $A$ and $B$ is large and replacing each such sandwich $(A,B)$ with finitely many sandwiches that have a smaller discrepancy between the inner and the outer part.
For the quantification of the discrepancy between~$A$ and~$B$ one can employ different functions.

For our purposes it is natural to use the \emph{lattice point gap} $\card{B \cap \Z^m} - \card{A \cap \Z^m}$.
A natural approach to replace $(A,B)$ by sandwiches with a smaller lattice point gap is to pick a vertex $v$ of $B$ that is not contained in $A$ and modify $(A,B)$ to two sandwiches: one with the inner part containing $v$ and one with the outer part not containing $v$.
The iterative procedure continues until all sandwiches in the list have lattice point gap equal to zero.

There are two important aspects that allow to optimize the running time.
Two sandwiches $(A,B)$ and $(A',B')$ are called \emph{equivalent} if there is a unimodular transformation that simultaneously brings $A$ to $A'$ and $B$ to $B'$.
Thus, for the enumeration of polytopes with a property $\sfP$ that is invariant up to unimodular equivalence, it is sufficient to keep sandwiches up to this notion of equivalence.
Our enumeration task concerns the property $\sfP$ of a lattice polytope~$A$, describing that $\Delta(A) = \Delta$.
As described in~\cite[Lem.~7.9]{averkovborgersoprunov2021classification}, equivalence of two given sandwiches can be expressed as unimodular equivalence of suitable higher-dimensional lattice polytopes associated with the two sandwiches.
The second aspect that allows to optimize the running time is \emph{monotonicity}.
If~$\sfP$ is the conjunction $\sfP = \sfP_1 \wedge \sfP_2$, where $\sfP_1$ is a downward closed property, while~$\sfP_2$ is an upward closed property, we can prune those sandwiches $(A,B)$ that are generated for which $A$ does not satisfy $\sfP_1$, or $B$ does not satisfy $\sfP_2$.

A further tool for an efficient implementation of these ideas is the
\emph{reduction of a polytope $B$} relative to some polytope $A \subseteq B$.
This simply means, that before adding a possible new sandwich $(A,B)$ during the iteration, we neglect all integer points~$v \in B$ such that $\Delta(A \cup \{v\}) > \Delta(A)$.
More precisely, the \emph{reduced sandwich} $(A,B')$ of $(A,B)$ is defined by $B' = \conv\{ v \in B \cap \Z^m : \Delta(A \cup \{v\}) = \Delta(A) \}$.

With these details of the implementation in mind, we can now describe the procedure to solve \Cref{prob:classification-full} as done in \Cref{algo:sandwich-full}.
Regarding the classification of all $o$-symmetric lattice polytopes $P \subseteq \R^m$ with $\Delta(P) = \Delta$ and $\h(\Delta,m) = \card{P \cap \Z^m}$ in \Cref{prob:classification-maximal}, we need to make the following adjustments to \Cref{algo:sandwich-full}:

\begin{algorithm*}[hb]
\caption{Sandwich Factory Algorithm that solves \Cref{prob:classification-full}}
\label{algo:sandwich-full}
\begin{algorithmic}[0]
\REQUIRE{A dimension $m \in \Z_{>0}$ and a value $\Delta \in \Z_{>0}$.}
\ENSURE{A list of all full-dimensional $o$-symmetric lattice polytopes $P \subseteq \R^m$ with $\Delta(P) = \Delta$, up to unimodular equivalence.}
\STATE{\medskip\textbf{Step 1:} Initialization}
\STATE{
\begin{itemize}
 \item enumerate $o$-symmetric lattice crosspolytopes $A \subseteq \R^m$ with $\Delta(A) = \Delta$
 \item for each $A = \conv\{\pm v_1,\ldots,\pm v_m\}$ as above choose
 \[
 Q := \left\{ \sum_{i=1}^m \alpha_i v_i : -1 \leq \alpha_i \leq 1, \forall i \in [m]\right\}
 \]
 as the lattice parallelepiped spanned by the vertices of~$A$
 \item $B := $ reduction of $Q$ relative to $A$
 \item initialize the sandwich factory $\cF$ with all pairs $(A,B)$ obtained above 
\end{itemize}}
\STATE{\medskip\textbf{Step 2:} Iterative reduction of maximal lattice point gap}
\WHILE{there are sandwiches with a positive lattice point gap}
 \STATE{
 \setlist[itemize]{leftmargin=20pt}
 \begin{itemize}
  \item $(A,B) :=$ a sandwich with maximal lattice point gap
  \item $v :=$ a vertex of $B$ that is not contained in $A$
  \item $A' := \conv\{A \cup \{\pm v\}\}$
  \item $B' := $ reduction of $B$ relative to $A'$
  \item $B'' := \conv\{(B \cap \Z^m) \setminus \{\pm v\}\}$
  \item add $(A',B')$ to $\cF$, if $\cF$ does not already contain a sandwich that is equivalent to $(A',B')$
  \item add $(A,B'')$ to $\cF$, if $\cF$ does not already contain a sandwich that is equivalent to $(A,B'')$ 
 \end{itemize}
 }
\ENDWHILE
\STATE{\medskip\textbf{Step 3:} Return the results}
\STATE{
\begin{itemize}
 \item all sandwiches $(A,B)$ in $\cF$ have now the form $A = B$
 \item return the set of all~$A$ such that $(A,B) \in \cF$
\end{itemize}
}
\end{algorithmic}
\end{algorithm*}

\begin{enumerate}[label=(\roman*)]
 \item We maintain a value $\mathrm{cmax}$ that we initialize with the valid lower bound $m^2 + m + 1 + 2m(\Delta-1)$ on $\h(\Delta,m)$ (see~\eqref{eqn:lee-et-al-bound}).
 \item We never add a sandwich $(A,B)$ during the algorithm with $\card{B \cap \Z^m} < \mathrm{cmax}$.
 \item If we add a sandwich $(A,B)$ to $\cF$ with $\card{A \cap \Z^m} > \mathrm{cmax}$, then we update $\mathrm{cmax}$ to $\card{A \cap \Z^m}$.
\end{enumerate}

\begin{remark}\
\begin{enumerate}[label=(\roman*)]
  \item The initialization of the sandwich factory in Step 1 of \Cref{algo:sandwich-full} can be done by generating all Hermite normal forms of integer matrices $M \in \Z^{m \times m}$ with $\Delta(M) = \Delta$ (cf.~Schrijver~\cite[Sect.~4.1]{schrijver1986theory}).
  
  \item One can interpret our sandwich factory based approach to \Cref{prob:classification-maximal} as a branch-and-bound procedure for the maximization of $|P \cap \Z^m|$ subject to $\Delta(P) \le \Delta$.
  In fact, each sandwich $(A,B)$ corresponds to a node of a branch-and-bound tree.
Replacing a sandwich with two new sandwiches is branching, and the removal of sandwiches in adjustment~(ii) is pruning of a node from the branch-and-bound tree. 

\end{enumerate}
\end{remark}

\subsection{Computational Results}

We have implemented the previously described algorithms in \texttt{sagemath}~\cite{sage}, based on the existing implementation of the sandwich factory used in~\cite{averkovborgersoprunov2021classification} by Christopher Borger\footnote{see \url{https://github.com/christopherborger/mixed_volume_classification}}.
The source code as well as data files containing the results of our computations are available at \url{https://github.com/mschymura/delta-classification}.

The computational results regarding the constant $\h(\Delta,m)$ and the number of equivalence classes of $o$-symmetric lattice polytopes for a given~$\Delta$ are gathered in Tables~\ref{table:h-delta-m},\ref{table:number-delta-polytopes} and~\ref{table:number-extremal-delta-polytopes}.

\begin{table}[ht]
\centering
\begin{tabular}{|c||c|c|c|c|c|c|c|c|c|c|c|}
\hline
$m \ \backslash\ \Delta$ & $1$ & $2$ & $3$ & $4$ & $5$ & $6$ & $7$ & $8$ & $9$ & $10$ & $11$ \\ \hline\hline
$2$                   &  $7$   &  $11$   &  $15$   &  $19$  &  $23$   &  $27$   &  $31$   &  $35$   &  $39$   &  $43$  & $47$   \\ \hline
$3$                   &  $13$   &  $19$   &  $\mathbf{25}$   &  $\mathbf{33}^*$   &  $\mathbf{37}$   &  $\mathbf{43}$   &  $\mathbf{49}$   &  $\mathbf{55}$   &  $\mathbf{61}$   &  $\mathbf{67}$  & $\mathbf{73}$ \\ \hline
$4$                   &  $21$   &  $29$   &     &  $\mathbf{\geq 49^*}$   &     &     &     &  $\mathbf{\geq 81^*}$   &     &   &   \\ \hline
$5$                   &  $31$   &  $41$   &     &  $\mathbf{\geq 67^*}$   &     &     &     &   $\mathbf{\geq 109^*}$  &     &   &   \\ \hline
\end{tabular}
\caption{The values of $\h(\Delta,m)$ for small numbers $m,\Delta$. The values in bold have not been known before. Values with an asterisk $^*$ indicate that $\h(\Delta,m)$ is larger than it was conjectured by Lee et al. (see \Cref{conj:lee-exact-value-hDm}).}
\label{table:h-delta-m}
\end{table}

\Cref{table:h-delta-m} determines the previously unknown exact values of $\h(\Delta,m)$, for $m=3$ and $3 \leq \Delta \leq 11$.
It also reveals a counterexample to \Cref{conj:lee-exact-value-hDm} for the case $(\Delta,m) = (4,3)$, whose structure we used to construct counterexamples for every $(\Delta,m) \in \{ (4,4) , (4,5) , (8,4) , (8,5) \}$ as well.
The construction is discussed further below in the next section.

\Cref{table:number-delta-polytopes} reports on the classification \Cref{prob:classification-full} and lists the number of equivalence classes of $o$-symmetric lattice polytopes $P \subseteq \R^m$ with $\Delta(P) = \Delta$, and for the parameters~$\Delta,m$ for which our algorithm stopped within at most~$3$ days of running time.

\begin{table}[ht]
\centering
\begin{tabular}{|c||c|c|c|c|c|c|c|c|c|c|c|}
\hline
$m \ \backslash\ \Delta$ & $1$ & $2$ & $3$ & $4$ & $5$ & $6$ & $7$ & $8$ & $9$ & $10$ & $11$ \\ \hline\hline
$2$                   &  $2$   &  $4$   &  $7$   &  $15$  &  $16$   &  $35$   &  $29$   &  $72$   &  $77$   &  $126$  & $99$  \\ \hline
$3$                   &  $5$   &  $32$   &  $102$   &  $554$   &  $996$  &  $5937$   &  $8029$   &    &     &  &    \\ \hline
$4$                   &  $17$   &  $448$  &     &     &     &     &     &     &     &   &   \\ \hline
\end{tabular}
\caption{The number of equivalence classes of $o$-symmetric lattice polytopes $P \subseteq \R^m$ for given~$\Delta(P) = \Delta$.}
\label{table:number-delta-polytopes}
\end{table}

For the pairs $(\Delta,3)$ with $8 \leq \Delta \leq 11$, we used the modification of \Cref{algo:sandwich-full} that solves \Cref{prob:classification-maximal}.
The corresponding number of equivalence classes of \emph{extremizers} of~$\h(\Delta,m)$ are given in \Cref{table:number-extremal-delta-polytopes}.
It is interesting to observe that, starting from dimension $m=3$, non-uniqueness of an extremizer of~$\h(\Delta,m)$ is the norm.

\begin{table}[ht]
\centering
\begin{tabular}{|c||c|c|c|c|c|c|c|c|c|c|c|}
\hline
$m \ \backslash\ \Delta$ & $1$ & $2$ & $3$ & $4$ & $5$ & $6$ & $7$ & $8$ & $9$ & $10$ & $11$ \\ \hline\hline
$2$                   &  $1$   &  $1$   &  $1$   &  $2$  &  $1$   &  $1$   &  $1$   &  $1$   &  $1$   &  $1$  & $1$  \\ \hline
$3$                   &  $1$   &  $3$   &  $3$   &  $1$   &  $5$  &  $7$   &  $8$   &  $12$  &  $13$   & $14$ &  $16$   \\ \hline
$4$                   &  $1$   &  $2$  &     &     &     &     &     &     &     &   &   \\ \hline
\end{tabular}
\caption{The number of equivalence classes of $o$-symmetric lattice polytopes $P \subseteq \R^m$ for given~$\Delta(P) = \Delta$, and which satisfy $\h(\Delta,m) = \card{P \cap \Z^m}$.}
\label{table:number-extremal-delta-polytopes}
\end{table}

%
%
%

\subsection{Construction of Counterexamples and {\Cref{thm:counterexamples-Delta-4}}}

For a finite set $S \subseteq \R^m$, we denote by
\[
\cD(S) := S - S = \left\{ a - b : a,b \in S \right\}
\]
its \emph{difference set}, and by
\[
\pyr(S) := \left( S \times \{0\} \right) \cup \{ e_{m+1} \} \subseteq \R^{m+1}
\]
the \emph{pyramid} over~$S$ of height one.
It turns out that a combination of these two operations behaves very well with respect to the maximal absolute value of the determinant of $m$-element subsets.

\begin{lemma}
\label{lem:delta-of-pyramids}
For every finite non-empty set $S \subseteq \R^m$, we have
\[
\card{\cD(\pyr(S))} = \card{\cD(S)} + 2 \card{S} \quad \textrm{and} \quad \Delta(\cD(\pyr(S))) = \Delta(\cD(S)).
\]
\end{lemma}

\begin{proof}
The difference set of the pyramid over~$S$ is given by
\[
\cD(\pyr(S)) = \left( (S - S) \times \{0\} \right) \cup \left( -S \times \{1\} \right) \cup \left( S \times \{-1\} \right).
\]
Since $S$ is finite, this immediately yields the cardinality count.

So, let's prove the statement on the largest $m \times m$ minors in $\cD(\pyr(S))$.
First of all, we have $\Delta(\cD(\pyr(S))) \geq \Delta(\cD(S))$, because we can take an $m$-element subset $S' \subseteq \cD(S)$ with $\card{\det(S')} = \Delta(\cD(S))$ and lift this to the set $S'' := (S' \times \{0\}) \cup \{ (-s,1) \} \subseteq \cD(\pyr(S))$, for some $s \in S$.
Clearly, $\card{\det(S'')} = \card{\det(S')} = \Delta(\cD(S))$.

Conversely, let $S' = \{s_0',s_1',\ldots,s_m'\} \subseteq \cD(\pyr(S))$ be a subset of size $m+1$, and without loss of generality, let $s_0',s_1',\ldots,s_\ell' \in \left( (S-S) \times \{0\} \right)$ be with $s_i' = (s_{i1}-s_{i2},0)$, for suitable $s_{i1},s_{i2} \in S$, and $s_{\ell+1}',\ldots,s_m' \in \left( -S \times \{1\} \right)$ be with $s_j' = (-s_j,1)$, for suitable $s_j \in S$.
With this notation, we obtain
\begin{align*}
\card{\det(S')} &= \card{\det\left(\tbinom{s_{01}-s_{02}}{0},\ldots,\tbinom{s_{\ell1}-s_{\ell2}}{0},\tbinom{-s_{\ell+1}}{1},\ldots,\tbinom{-s_{m}}{1}\right)}\\
&= \card{\det\left(\tbinom{s_{01}-s_{02}}{0},\ldots,\tbinom{s_{\ell1}-s_{\ell2}}{0},\tbinom{s_m-s_{\ell+1}}{0},\ldots,\tbinom{s_m-s_{m-1}}{0},\tbinom{-s_{m}}{1}\right)}\\
&= \card{\det\left(s_{01}-s_{02},\ldots,s_{\ell0}-s_{\ell1},s_m-s_{\ell+1},\ldots,s_m-s_{m-1}\right)}\\
&\leq \Delta(\cD(S)).\qedhere
\end{align*}
\end{proof}

Based on this lemma we can now construct series of examples that exceed the conjectured value of~$\h(4,m),\h(8,m)$, and $\h(16,m)$ in Conjecture~\ref{conj:lee-exact-value-hDm} by an additive term that is linear in~$m$, and for~$m$ large enough.
Our construction is based on the unique (up to unimodular equivalence) set attaining the value $\h(4,3) = 33$, resulting from our enumeration approach described earlier (see \Cref{table:h-delta-m,table:number-extremal-delta-polytopes}).
This set can be written as
\[
\cD(\pyr(H_2)) = \left( 2 \cdot H_2 \times \{0\} \right) \cup \left( H_2 \times \{-1,1\} \right),
\]
where $H_2 = \cD(\{\zero,e_1,e_2\}) = \left\{ \zero, \pm e_1, \pm e_2, \pm (e_1 - e_2) \right\} \subseteq \Z^2$ is the two-dimensional set attaining the Heller constant $\h(1,2) = 7$.
Iterating this pyramid construction and using the $\ell$-dimensional set $H_\ell = \cD(\{\zero,e_1,\ldots,e_\ell\})$ attaining the Heller constant $\h(1,\ell) = \ell^2 + \ell + 1$, we define
\[
C_m^\ell := \pyr^{m-\ell}(H_\ell) = \underbrace{\pyr(\ldots\pyr(\pyr(H_\ell))\ldots)}_{m-\ell \textrm{ times}} \subseteq \Z^m.
\]

\begin{proof}[Proof of \Cref{thm:counterexamples-Delta-4}]
The point sets yielding the claimed lower bounds are of the form~$\cD(C_m^\ell)$, that is, the difference set of~$C_m^\ell$.
In view of \Cref{lem:delta-of-pyramids}, we get
\[
\Delta(\cD(C_m^\ell)) = \Delta(\cD(C_{m-1}^\ell)) = \ldots = \Delta(\cD(H_\ell)) = \Delta(2 \cdot H_\ell) = 2^\ell,
\]
and, using that the pyramid construction adds exactly one point to the set that the pyramid is taken over, we also obtain $\card{C_m^\ell} = m - \ell + \card{H_\ell}$.
Hence, using \Cref{lem:delta-of-pyramids} once more, we get
\begin{align}
\card{\cD(C_m^\ell)} &= \card{\cD(C_{m-1}^\ell)} + 2\card{C_{m-1}^\ell} = \card{\cD(C_{m-2}^\ell)} + 2\card{C_{m-2}^\ell} + 2\card{C_{m-1}^\ell} \nonumber\\
&= \card{\cD(H_\ell)} + 2 \sum_{i=\ell}^{m-1} \card{C_i^\ell} = \card{\cD(H_\ell)} + 2 \sum_{i=\ell}^{m-1} (i - \ell + \card{H_\ell}) \nonumber\\
&= \card{\cD(H_\ell)} + 2 \sum_{j=0}^{m-\ell-1} (\card{H_\ell} + j) \nonumber\\
&= \card{\cD(H_\ell)} + 2(m-\ell)\card{H_\ell} + (m-\ell)(m-\ell-1) \nonumber\\
&= m^2 + \left(2 \card{H_\ell} - 2 \ell - 1 \right) m + \left( \card{\cD(H_\ell)} - 2\ell\card{H_\ell} + \ell(\ell+1) \right).\label{eqn:counterexamples}
\end{align}
The conjectured value in \Cref{conj:lee-exact-value-hDm} is $\h(2^\ell,m) = m^2 + (2^{\ell+1} - 1) m + 1$.
Using $\card{H_\ell} = \ell^2 + \ell + 1$ and~\eqref{eqn:counterexamples}, this means that $\cD(C_m^\ell)$ is a counterexample to \Cref{conj:lee-exact-value-hDm}, for fixed $\ell$ and large enough~$m$, if and only if
\[
2 \card{H_\ell} - 2 \ell - 1 = 2\ell^2 + 1 > 2^{\ell+1} - 1.
\]
This holds exactly for $\ell \in \{2,3,4\}$, and computing $\card{\cD(H_2)} = 19$, $\card{\cD(H_3)} = 55$, and $\card{\cD(H_4)} = 131$, we get by~\eqref{eqn:counterexamples}
\begin{align*}
\h(4,m) &\geq \card{\cD(C_m^2)} =  m^2 + 9m - 3   &&= m^2 + 7m + 1 + 2(m-2), \\
\h(8,m) &\geq \card{\cD(C_m^3)} = m^2 + 19m - 11  &&= m^2 + 15m + 1 + 4(m-3), \\
\h(16,m) &\geq \card{\cD(C_m^4)} = m^2 + 33m - 17 &&= m^2 + 31m + 1 + 2(m-9),
\end{align*}
and the claim follows.
\end{proof}


\section{Open Problems}
\label{sect:open-problems}

The determination of the exact value of $\h(\Delta,m)$ remains the major open problem. Note that the bounds from other sources and the bound we prove here are incomparable when both $m$ and $\Delta$ vary.
In order to understand the limits of our method for upper bounding $\h(\Delta,m)$, it is necessary to determine the exact asymptotic behavior of $\hs(m)$ or $\hsd{\Delta}{m}$.
\Cref{prop:shifted-Heller-dim-2-and-3} suggests that the following may have an affirmative answer:

\begin{question}
It is true that for every $m \in \Z_{>0}$, we have $\hs(m) \leq \h(1,m)
$?
\end{question}

This would imply a bound of order $\h(\Delta,m) \in \cO(m^2) \cdot \Delta$, which is with a view at \Cref{prop:shifted-heller-lower-bound} the best possible we can achieve based on $\hs(m)$.

A relaxed question concerns the refined shifted Heller constant:

\begin{question}
Do we have $\hsd{\Delta}{m} \leq \h(1,m)$, for every $m,\Delta \in \Z_{>0}$?
\end{question}

Again an affirmative answer would imply the bound $\h(\Delta,m) \in \cO(m^2) \cdot \Delta$, but \Cref{prop:shifted-heller-lower-bound} does not rule out the possibility that $\hsd{\Delta}{m}$ actually grows sublinearly with~$m$, for fixed~$\Delta$, as the construction therein uses a translation vector with denominator equal to~$m$.

Relaxing the question once again, we may ask

\begin{question}
Is it true that $\h(\Delta,m) \in \cO(m^2) \cdot \Delta$?
\end{question}

Our computational experiments described in \Cref{sect:computations} suggest that there are more constraints on the maximal size of a $\Delta$-modular integer matrix, when $\Delta$ is a prime (compare also with the improved bound in \Cref{thm:polynomial-linear}~(ii) for odd~$\Delta$).

\begin{question}
Does \Cref{conj:lee-exact-value-hDm} hold for every $m \in \Z_{>0}$, and every \emph{prime} $\Delta \in \Z_{>0}$?
In particular, does it hold for $\h(3,m)$?
\end{question}

Moreover, from the data in \Cref{table:h-delta-m} one may also suspect that for any given $m \in \Z_{>0}$ there are only finitely many $\Delta \in \Z_{>0}$ that possibly violate \Cref{conj:lee-exact-value-hDm}.
For instance, it could very well be that the value $\h(4,3) = 33$ is the only exception from \Cref{conj:lee-exact-value-hDm} in dimension $m=3$.

\begin{question}
Given $m \in \Z_{>0}$, is there always a threshold $\Delta(m) \in \Z_{>0}$ such that $\h(\Delta,m) = m^2+m+1+2m(\Delta-1)$, for every $\Delta \geq \Delta(m)$?
\end{question}

Finally, while investigating the extreme examples attaining $\h(\Delta,m)$, for small values of $m$ and $\Delta$, and which are enumerated in \Cref{table:number-extremal-delta-polytopes}, we found that for each computed pair $(\Delta,m)$ there is at least one extremizer that can be written as the set of integer points in the convex hull of the difference set of some subset of~$\Z^m$.
We wonder whether this is a general phenomenon:

\begin{question}
Is there always an extremizer for $\h(\Delta,m)$ that can be written as the set of integer points in the convex hull of the difference set of a subset of~$\Z^m$?
\end{question}

\subsection*{Acknowledgments}
We thank Rudi Pendavingh for pointing us to the paper of Geelen et al.~\cite{geelennelsonwalsh2021excludingaline}.

\bibliographystyle{amsplain}
\bibliography{mybib}

\end{document}